\documentclass[11pt]{article}
\usepackage[top=1in, bottom=1in, left=1in, right=1in]{geometry}

\usepackage{amsmath,amssymb,amscd,amsthm,amsfonts}

\usepackage[section]{placeins}
\newcommand{\ms}[1]{\mathscr{#1}}

\usepackage{color}
\usepackage{bm,multirow,algorithm, algorithmic}
\usepackage{enumerate,mathrsfs}
\usepackage{graphicx}
\usepackage{epstopdf}
\usepackage{amsfonts}
\usepackage{ amssymb }
\usepackage[square,sort,comma,numbers]{natbib}

\usepackage{makecell}

\usepackage{amsmath}    
\usepackage[table]{xcolor}
\usepackage[T1]{fontenc}
\usepackage[utf8]{inputenc}

\usepackage{amsmath,color}
\allowdisplaybreaks

\usepackage{epstopdf}
\usepackage{subfigure}
\usepackage{amsfonts}
\usepackage{enumitem}

\allowdisplaybreaks

\usepackage{amsmath}    
\usepackage{amsmath, algorithm}
\usepackage[table]{xcolor}
\usepackage[T1]{fontenc}
\usepackage[utf8]{inputenc}

\usepackage{amsmath}
\usepackage{epsf}
\usepackage{mathtools}
\DeclarePairedDelimiter\autobracket{(}{)}
\newcommand{\br}[1]{\autobracket*{#1}}

\newcommand{\comm}[1]{{\color{red}#1}}

\DeclareMathOperator*{\argmax}{argmax}
\DeclareMathOperator*{\argmin}{argmin}
\DeclareMathOperator{\sign}{sign}
\newcommand{\abs}[1]{\lvert #1 \rvert}
\newcommand{\norm}[1]{\left\lVert #1 \right\rVert}
\newcommand{\ceil}[1]{\lceil #1 \rceil}
\newcommand{\floor}[1]{\lfloor #1 \rfloor}
\newcommand{\st}{\;:\;}                         
\newcommand{\ve}[2]{\langle #1 ,  #2 \rangle}   
\newcommand{\eqdef}{\stackrel{\text{def}}{=}}
\newcommand{\R}{\mathbb{R}}
\newcommand{\Reg}{\psi}
\newcommand{\Exp}{\boldsymbol{E}}
\newcommand{\Prob}{\boldsymbol{P}}
\newcommand{\calP}{\mathcal{P}}
\newcommand{\sA}{{\sigma_{\min}(A)}}
\newcommand{\sB}{{\sigma_{\max}(B)}}
\newcommand{\sG}{{\sigma_{\max}(G)}}
\newcommand{\smH}{{\sigma_{\min}(H)}}
\newcommand{\sMH}{{\sigma_{\max}({H})}}
\newcommand{\stH}{{\sigma_{\max}(\tilde{H})}}
\newcommand{\cG}{{\cal G}}
\newcommand{\expec}[1]{\mathbf{E}\left[ #1 \right] }
\newcommand\numberthis{\addtocounter{equation}{1}\tag{\theequation}}  
\newcommand\myeq{\stackrel{\mathclap{\tiny\mbox{def}}}{=}}
\newcommand{\zo}{{\emph{ZO}}}
\newcommand{\fo}{{\emph{FO}}}
\newcommand{\beq}{\begin{equation}}
\newcommand{\eeq}{\end{equation}}
\newcommand{\beqa}{\begin{eqnarray}}
\newcommand{\eeqa}{\end{eqnarray}}
\newcommand{\beqas}{\begin{eqnarray*}}
\newcommand{\eeqas}{\end{eqnarray*}}
\newcommand{\bi}{\begin{tightitemize}}
\newcommand{\ei}{\end{tightitemize}}
\newcommand{\ba}{\begin{array}}
\newcommand{\ea}{\end{array}}
\def\eqnok#1{(\ref{#1})}

\DeclareMathOperator{\dom}{dom}
\DeclareMathOperator{\Tr}{Tr}

\newtheorem{definition}{Definition}[section]
\newtheorem{lemma}{Lemma}[section]
\newtheorem{theorem}{Theorem}[section]
\newtheorem{proposition}{Proposition}[section]

\newtheorem{assumption}{Assumption}[section]
\newtheorem{remark}{Remark}

\newcommand{\defeq}{\stackrel{def}{=}}
\def\endproof{{\ \hfill\hbox{%
     \vrule width1.0ex height1.0ex    }\parfillskip 0pt}\par}

\def\vgap{\vspace*{.05in}}
\def\E{{\bf E}}
\def\kb{{\textcolor{blue}{\bf KB:}}}
\def\ar{{\textcolor{green}{\bf AR:}}}
\def\exp{{\rm exp}}
\def\prob{{\rm Prob}}
\def\setU{{X}}
\def\hu{{u^*}}
\def\uh{{h_X}}
\def\tu{{\tilde x}}
\def\usigma{{\sigma_X}}
\newcommand{\bbe}{\Bbb{E}}
\def\prob{\mathop{\rm Prob}}
\def\Prob{{\hbox{\rm Prob}}}
\newcommand{\bbr}{\Bbb{R}}
\def\w{\omega}
\def\sG{\tilde {\cal H}}
\def\rint{{\rm rint \,}}
\def\lin{{\rm lin \,}}
\def\cA{{\cal A}}
\def\SO{{\cal SO}}
\def\sigmasco{{\sigma}}
\def\lipsh{{M}}
\def\Optsco{{\Psi^*}}
\def\half{{\textstyle{\frac{1}{2}}}}
\def\cB{{\cal B}}
\def\cC{{\cal C}}
\def\cD{{\cal D}}
\def\cP{{\cal P}}
\def\cG{{\cal G}}
\def\cR{{\cal R}}
\def\cQ{{\cal Q}}
\def\cS{{\cal S}}
\def\optcCE{{\cal C}_0^s}
\def\optcCP{{\cal C}_1^s}
\def\tcC{{\tilde{\cal C} }}
\def\cH{{\cal H}}
\def\cX{{\cal X}}
\def\cY{{\cal Y}}
\def\cI{{\cal I}}
\def\cF{{\cal F}}
\def\cK{{\cal K}}
\def\cU{{\cal U}}
\def\tx{{\tilde x}}
\def\ty{{\tilde y}}
\def\eini{{\cal V}}
\def\einis{{\cal V}}
\def\lb{{\rm lb}}
\def\ub{{\rm ub}}
\def\sd{{\rm sd}}
\def\tr{{\rm tr}}
\def\dom{{\rm dom}}
\def\risk{{l}}
\def\regu{{r}}
\def\zg{{G_t\left( x_t,u_t\right) }}
\def\zh{{H_t\left( x_t,u_t\right) }}
\def\bg{{\bar{G}_t}}
\def\bh{{\bar{H}_t}}
\usepackage{amsmath}
\usepackage{amsthm}
\usepackage{amsfonts}
\usepackage{amssymb}
\usepackage{hyperref}
\usepackage[noabbrev,capitalise,nameinlink]{cleveref}
\hypersetup{colorlinks={true},linkcolor={blue},citecolor=blue}

\crefname{assumption}{Assumption}{assumptions}

\Crefname{table}{Table}{Tables}

\crefname{equation}{}{}

\title{Stochastic Multi-level Composition Optimization Algorithms with Level-Independent Convergence Rates\thanks{Authors are listed by alphabetical order.}}

\author{Krishnakumar Balasubramanian\thanks{Department of Statistics, University of California, Davis. \texttt{kbala@ucdavis.edu}.}
\and Saeed Ghadimi\thanks{Department of Management Sciences, University of Waterloo. \texttt{sghadimi@uwaterloo.ca}.}
\and Anthony Nguyen\thanks{Department of Mathematics, University of California, Davis. \texttt{anthonynguyen@math.ucdavis.edu}. }
}

\begin{document}
\maketitle
\begin{abstract}
In this paper, we study smooth stochastic multi-level composition optimization problems, where the objective function is a nested composition of $T$ functions. We assume access to noisy evaluations of the functions and their gradients, through a stochastic first-order oracle. For solving this class of problems, we propose two algorithms using moving-average stochastic estimates, and analyze their convergence to an $\epsilon$-stationary point of the problem. We show that the first algorithm, which is a generalization of~\cite{GhaRuswan20} to the $T$ level case, can achieve a sample complexity of $\mathcal{O}_T(1/\epsilon^6)$ by using mini-batches of samples in each iteration, where $\mathcal{O}_T$ hides constants that depend on $T$. By modifying this algorithm using linearized stochastic estimates of the function values, we improve the sample complexity to $\mathcal{O}_T(1/\epsilon^4)$. {\color{black}This modification not only removes the requirement of having a mini-batch of samples in each iteration, but also makes the algorithm parameter-free and easy to implement}. To the best of our knowledge, this is the first time that such an online algorithm designed for the (un)constrained multi-level setting, obtains the same sample complexity of the smooth single-level setting, under standard assumptions (unbiasedness and boundedness of the second moments) on the stochastic first-order oracle.
\end{abstract}

\section{Introduction}\label{sec:intro}
We consider multi-level stochastic composition optimization problems of the form
\begin{align}\label{eq:mainprob}
\min_{x \in X} \Big\{  F(x) = f_1 \circ \cdots \circ f_T(x)  \Big\},
\end{align}
where $f_i:\mathbb{R}^{d_i} \to \mathbb{R}^{d_{i-1}}$ for $i = 1, \dots, T$ ($d_0 = 1$) are continuously differentiable functions, {\color{black}the composite function $F$ is bounded below by $F^* > -\infty$}, and
$X$ is a closed convex set. We assume that the exact values and derivatives of $f_i$'s are not available. In particular, we assume that $f_i(y) = \mathbb{E}_{\xi_i}[G_i(y,\xi_i)]$ for some random variables $\xi_i \in \mathbb{R}^{\tilde{d_i}}$.
Note that when $T=1$, the problem reduces to the standard stochastic optimization problem which has been well-explored in the literature; see, for example~\cite{robbins1951stochastic, kushner2003stochastic, borkar2009stochastic, GhaLan13, GhaLan16, shapiro2014lectures}, for a partial list. In this work, we consider stochastic first-order algorithms for solving~\cref{eq:mainprob} when $T\geq 1$.  Note that the gradient of the function $F(x)$ in~\cref{eq:mainprob}, is $\nabla F(x) = \nabla f_T(y_T)\nabla f_{T-1}(y_{T-1}) \cdots \nabla f_1(y_1) $, {\color{black}where $\nabla f_i$ denotes the transpose of the Jacobian of $f_i$}, $y_i = f_{i+1} \circ \cdots \circ f_T(x)$ for $1 \leq i < T$, and $y_T = x$. Our goal is to solve the above optimization problem, given access to noisy  evaluations of $\nabla f_i$'s and $f_i$'s. Precise assumptions on our stochastic first-order oracle considered will be stated later in~\cref{sec:originalnasa}. Because of the nested nature of the gradient $\nabla F(x)$, obtaining an unbiased gradient estimator in the online setting, with controlled higher moments, becomes non-trivial.

Although problems of the form in~\cref{eq:mainprob} have been considered since the work of~\cite{Ermolievold}, recently there has been a renewed interest in this problem due to applications arising in mathematical finance, nonparametric statistics, deep generative modeling and reinforcement learning. We refer the reader to~\cite{ermoliev2013sample, wang2017stochastic, dentcheva2017statistical, blanchet2017unbiased, wang2016accelerating,bora2017compressed, yang2019multi-level, GhaRuswan20, zhang2019multi, ongie2020deep} for such applications and various algorithmic approaches for solving problem~\cref{eq:mainprob}. In particular~\cite{wang2017stochastic} and~\cite{yang2019multi-level} considered the case of $T=2$ and general $T$ respectively, and analyzed stochastic gradient-type algorithms. Such an approach leads to level-dependent and sub-optimal convergence rates. However, large deviation and Central Limit Theorem results established in~\cite{ermoliev2013sample} and~\cite{dentcheva2017statistical}, respectively, show that in the sample-average or empirical risk minimization setting, the $\argmin$ of the problem in~\cref{eq:mainprob} based on $n$ samples, converges at a level-independent rate (i.e., dependence of the convergence rate on the target accuracy is independent of $T$) to the true minimizer, under suitable regularity conditions. Hence, it is natural to ask the following question: \emph{Is it possible to construct iterative online algorithms for solving problem~\cref{eq:mainprob} with level-independent convergence rates?} Recently, for the case of $T=2$,~\cite{GhaRuswan20} proposed a single time-scale Nested Averaged Stochastic Approximation (NASA) algorithm. The authors showed that by modifying the specific Lyapunov function, defined in \cite{ruszczynski1987linearization} for nonsmooth single-level stochastic optimization, the convergence analysis of the NASA algorithm can be established such that its complexity bound matches the case of $T=1$. This resolved the above question for $T=2$. However, constructing similar algorithms for the case of general $T$ remained less investigated.
\vgap

\textbf{Main contributions.} In this work, we propose two algorithms for solving problem \eqnok{eq:mainprob} with level-independent convergence rates in the stochastic first-order oracle setting, under mild assumptions. Our algorithms are applicable to both unconstrained and constrained cases, as we do not make any boundedness assumption on the feasible set $X$. Their complexity results are summarized in~\cref{table:summary}. The first algorithm is based on an extension of the NASA algorithm from~\cite{GhaRuswan20} (proposed for the case of $T=2$) to the general $T\geq1$ setting, requiring a mini-batch of sample in each iteration. Although this algorithm has level-independent convergence rates, the sample complexity (i.e., the number of calls to stochastic first-order oracle) does not match that of standard stochastic gradient algorithm for $T=1$~\text{or} the NASA algorithm for $T=2$. The second algorithm is based on a modification to the NASA algorithm by adding a linear bias correction term in evaluating the inner function values, motivated by the recent work \cite{rusz20} for nonsmooth multi-level composition problems.
For any $T \geq 1$, we show that this algorithm has the same oracle complexity as that of the regular stochastic gradient algorithm for the case of $T=1$, thereby providing a complete answer to the question above. We emphasize that unlike our first algorithm, this algorithm does not require a mini-batch of samples in any iteration and hence is more suitable to the online setting. Furthermore, it works with any positive constant step-size parameter choice (independent of problem parameters), thus making it easy to implement.
\vgap

\textbf{Comparisons to related works.}
A summary of our results, in comparison to the most related work of~\cite{yang2019multi-level} is provided in~\cref{table:summary}. We use $\mathcal{O}(\cdot)$ to represent the fact that the constants involved are only numerical constants that are independent of $T$. However, when the constants involved are dependent on $T$, we use $\mathcal{O}_T(\cdot)$.

The approach and the results in~\cite{yang2019multi-level} are provided only for the unconstrained setting. \textcolor{black}{Furthermore,~\cite{yang2019multi-level} requires an additional bounded fourth moment assumption on the stochastic Jacobian matrices. In an earlier version of our work uploaded to arXiv, we also made the same assumption, which however, we do not require in this work, thereby widening the applicability of the proposed method. } We also highlight the related work of~\cite{zhang2019multi} which considered problems of the form $\min_{x \in \mathbb{R}^{d_T}}\{ F(x) + H(x)\}$, with $F(x)$ being a multi-level composite function as in~\cref{eq:mainprob} and $H(x)$ being a convex and lower-semi-continuous function. Typically $H(x)$ could be considered as an indicator function of the constrained set $X$ to relate the above problem to our setup in~\cref{eq:mainprob}. The algorithm proposed in~\cite{zhang2019multi} is a proximal variant of SPIDER variance reduction technique~\cite{fang2018spider} and is a double-loop algorithm. Hence, it is predominantly applicable for finite-sum problems and is not so suitable for the general online problems that we focus on. Indeed, they assume that for a fixed batch of samples, one could query the oracle on different points, which is not suited for the general online stochastic optimization setup. Furthermore,~\cite{zhang2019multi} assume a much stronger mean-square Lipschitz smoothness assumption on the individual functions $f_i$ and their gradients, to obtain a complexity bound of $\mathcal{O} \left(T^6\rho^T/\epsilon^{3}\right)$, where $\rho$ is a problem dependent constant factor. To obtain their result, they also need a mini-batch of samples, with batch sizes of the order $T^3\rho^T$, which makes their approach impractical to use even for moderately large values of $T$. As mentioned above, our second algorithm does not have any such requirements, making it easy to be practically applicable for large values of $T$. 

As mentioned above, our~\cref{alg:modifiednasa} is related to a concurrent work \cite{rusz20}. In this work, the author focuses on nonsmooth multi-level composition problems and provides asymptotic convergence of the proposed algorithm to a stationary point of the problem by analyzing a system of differential inclusions which requires the compactness of the feasible set $X$. By further making the smoothness assumption, the author also establishes a sample complexity of ${\cal O}_T(1/\epsilon^4)$, similar to that of ~\cref{alg:modifiednasa} in \cref{optimal_theorem}. However, our convergence analysis here is distinct since we do not require the boundedness assumption of the feasible set which makes our method applicable to both unconstrained constrained problems.\footnote{We also remark that the finite-time convergence analysis of \cite{rusz20}, from our communication with the author, was not complete in the first version released on arXiv. However, more recently, after release of the first version of our paper on arXiv, the author has refined the convergence analysis in \cite{rusz20}.}

After our first draft appeared on arXiv,~\cite{chen2021solving} also proposed an approach for stochastic multi-level compositional optimization problems and obtained similar rates as us, albeit only for unconstrained problems and under the stronger assumption that the stochastic functions $G_i(y,\xi_i)$ are Lipschitz, almost surely.\footnote{It is worth noting that \cite{chen2021solving} was released several months after the first draft of  \cite{rusz20}.}


\begin{table}[t]
\centering
\begin{tabular}{|c|c|c|c|c|}
\cline{1-4}
Method & \cite{yang2019multi-level} & ~\cref{alg:originalnasa}  & ~\cref{alg:modifiednasa} \\ \hline
Convergence Rate& $\mathcal{O}_T\left(N^{-4/(7+T)}\right)$  &  \multicolumn{2}{c|}{$\mathcal{O}_T\left( N^{-1/2}\right)$}  \\ \hline
Oracle Complexity & $ \mathcal{O}_T\left(1/\epsilon^{(7+T)/2}\right)$&  $\mathcal{O}_T\left(1/\epsilon^{6}\right) $& $\mathcal{O}_T \left(1/\epsilon^{4}\right) $ \\ \hline
Mini-batch &No &Yes & No\\ \hline
Feasible Set & $X= \bbr^d$& \multicolumn{2}{c|}{General case} \\ \hline
Oracle Assumption&Finite $4$th moment&\multicolumn{2}{c|}{Finite $2$nd moment}  \\ \hline
\end{tabular}
\caption{Convergence rates and Oracle complexity results for finding an $\epsilon$-pair $(\bar{x},\bar{z})$ of ~\cref{eq:mainprob}; see \cref{def_Vxz} for details. Convergence rate refers to the upper bound on $\mathbb{E}[\sqrt{V(x,z)}]$ and oracle complexity refers to the number of calls to the stochastic first-order oracle to obtain a $\epsilon$-pair. The constants in~\cite{yang2019multi-level} and our work have exponential dependency on $T$ in the worst case. See ~\cref{remark:rmk1} and~\cref{remark:rmk2} for more details.}
\label{table:summary}
\end{table}

\subsection{Motivating Applications} We now provide two motivating applications of stochastic multi-level composition optimization problems. 

\subsubsection{Risk-averse Optimization}\label{sec:motivatingeg}
Our first motivating example to consider multilevel stochastic composite optimization problems is from the field of risk-averse stochastic optimization~\cite{ruszczynski2006optimization, yang2019multi-level}. Specifically, the mean-deviation risk-averse optimization is given by the following form:
\begin{align}\label{eq:mdriskaverse}
\min_x\left\{ \E[U(x,\xi)] + \lambda \left( \E \left[ \max \left\{ 0 , U(x,\xi) - \E[U(x,\xi)] \right\}^2  \right]\right)^{1/2} \right\}.
\end{align}
As noted in~\cite{yang2019multi-level, rusz20}, the problem in~\eqref{eq:mdriskaverse} is a  stochastic three-level composition optimization problem with
\begin{align*}
&f_3 (x) \coloneqq (x , \E[U(x,\xi)]),  \qquad f_2(y_3, y_4) \coloneqq \left(y_4,  \E \left[ \max \left\{ 0 , U(y_3,\xi) - y_4 \right\}^2  \right] \right), \\
& f_1(y_1,y_2) \coloneqq y_1 + \lambda  \sqrt{y_2+\delta}.
\end{align*}
Here, $\delta>0$ is added to make the square root function smooth. We consider a semi-parametric data generating process given by a single-index model of the form $b = g(a^\top x^*) +\zeta$, where $g:\mathbb{R} \to \mathbb{R}$ is called the link function. Such single-index models are widely used in statistics, machine learning and economics~\cite{ruppert2003semiparametric}. Here, $X$ is the input data which is assumed to be independent of the noise $\zeta$. The goal is to estimate the index $\beta^*$ in a risk-averse manner, as they are well-known to provide stable solutions~\cite{yang2019multi-level}. In this case, $\xi \coloneqq(a,b)$ and the function $U(x,\xi)$ depends on the loss function. We will revisit this example in Section~\ref{sec:exp} for numerical experiments. 

\subsubsection{Training large-scale Graph Neural Networks (GNNs)} Our second motivating example is training GNNs, which has been formulated as a stochastic multi-level compositional optimization problem in~\cite{cong2020minimal}. Each layer of a GNN is given by a matrix $H^{(i-1)} = L \sigma(H^{(i)})W^{(i)} \in \mathbb{R}^{n \times p}$, for $2\leq i \leq T$. Here, $L$ is the normalized graph Laplacian matrix (calculated as $D^{-1/2} A D^{-1/2}$ or $D^{-1}A$, where $D$ is the degree matrix and $A$ is the adjacency matrix given the data matrix $U \in \mathbb{R}^{n \times d}$), $W^{(i)}$ is the weight matrix at layer $i$, and $\sigma$ is the activation function, which either is a sigmoidal function $\sigma(s):=1/(1+e^{-s})$ or the ReLU function $\sigma(s):=\max\{ 0,s\}$ operating entry-wise on matrices. Furthermore, $H^{(T)} \coloneqq U$. When the size of the data set $n$ is large, subsampling methods are used to train the GNN~\cite{cong2020minimal}.

In our notation, the optimization variable $x\coloneqq\{W^{(1)}, \ldots, W^{(T-1)} \}$. The function $f_i$ is given by $f^{(i)} =L \sigma(H^{(i+1)})W^{(i)}$, for $i=2,\ldots T-1$, with $f^{(T)}\coloneqq U$. Furthermore, $f^{(1)}$ will be the user-defined loss function based on the label vector $Y\in\mathbb{R}^n$. The stochasticity in the problem is due to the fact that the data is subsampled when constructing the graph. Specifically, we have the random function given by $\tilde{L}^{(i-1)}\sigma(H^{(i)})W^{(i)}$, where $\tilde{L}^{(i-1)}$ is a stochastic approximation of the matrix $L$ such that $\E[\tilde{L}^{(i-1)}] = L$. We refer the interested reader to~\cite[Section 3]{cong2020minimal}, for additional details. We also remark that while the ReLU activation does not satisfy the smoothness assumptions we make in this work, the sigmoidal function does.

\vspace{0.1in}

\subsection{Organization} The rest of our paper is organized as follows. In \cref{sec:originalnasa}, we present our first algorithm and analyze its convergence for solving \cref{eq:mainprob} with any $T \ge 1$. In \cref{sec:modifiednasa}, we present a modification of this algorithm and show that it can recover the best-known sample complexity for (single-level) smooth stochastic optimization. In \cref{sec:exp}, we present some numerical experiments and conclude the paper with some remarks in \cref{conc_sec}.

\section{Multi-level Nested Averaging Stochastic Gradient Method}\label{sec:originalnasa}
In this section, we present our first algorithm for solving problem \eqnok{eq:mainprob}. As mentioned in~\cref{sec:intro}, the previously proposed stochastic gradient-type methods suffer in terms of the convergence rates when applied for solving this problem~\cite{yang2019multi-level}. The main reason is the increased bias when estimating the stochastic gradient of $F$, for $T \ge 2$. Our proposed algorithm has a multi-level structure -- in addition to estimating the gradient of $F$, we also estimate the values of inner functions $f_i$ by a mini-batch moving average technique, extending the approach in~\cite{GhaRuswan20} for any $T>1$. This will enable us to provide an algorithm with improved convergence rates to the stationary points compared to the prior work~\cite{yang2019multi-level}. Our approach is formally presented in~\cref{alg:originalnasa}.
\begin{algorithm}[ht]
\caption{Multi-level Nested Averaging Stochastic Gradient Method}
\begin{algorithmic}
	\STATE \textbf{Input:} \textcolor{black}{Positive integer sequences $\{b_k,\tau_k\}_{k \ge 0}$, step-size parameter $\beta$, and initial points $x^0 \in  X, \quad z^0 \in \bbr^{d_T}, \qquad w_i^0 \in \bbr^{d_{i-1}} \quad 1 \le i \le T$, and a probability mass function $P_R(\cdot)$ supported over $\{1,2,\ldots,N\}$, where $N$ is the number of iterations}.
	 \STATE  0. Generate a random integer number $R$ according to $P_R(\cdot)$.
    \FOR{$k = 0,1, 2, \dots, R$}
        \STATE  1. Compute
\begin{equation}
\label{def_uk}
u^k = \argmin_{y \in X}\  \left\{\langle z^k, y-x^k \rangle + \frac{\beta}{2} \|y-x^k\|^2\right\},
\end{equation}
stochastic gradients $J_i^{k+1}$, and function values $G_{i,j}^{k+1}$ at $w_{i+1}^k$ for $i=\{1,\dots,T\}, j=\{1,\dots,b_k\}$ by denoting $w_{T+1}^k \equiv x^k$.

\STATE 2. Set
\begin{align}
x^{k+1}  &= (1 - \tau_k)x^k + \tau_k u^k, \label{def_xk}\\
z^{k+1}  &= (1 - \tau_k)z^k + \tau_k \prod_{i=1}^TJ_{T+1-i}^{k+1}, \label{def_zk}\\
w_i^{k+1} &= (1 - \tau_k)w_i^k + \tau_k \bar G_i^{k+1}, \qquad 1 \le i \le T,\label{def_wi}
\end{align}
where
\beq\label{def_barG}
\bar G_i^{k+1} = \frac{1}{b_k}\sum_{j=1}^{b_k} G_{i,j}^{k+1}.
\eeq
    \ENDFOR
    \STATE \textbf{Output:} $(x^R, z^R, w_1^R, \ldots, w_T^R)$.
\end{algorithmic}
\label{alg:originalnasa}
\end{algorithm}

We now add a few remarks about \cref{alg:originalnasa}. First, note that at each iteration of this algorithm, we update the triple $(x^k, \{w^k\}_{i=1}^T, z^k)$, which are the convex combinations of the solutions to subproblem \eqnok{def_uk}, the estimates of inner function values $f_i$, and the stochastic gradient of $F$ at these points, respectively. It should be mentioned that we do not need to estimate the values of the outer function $f_1$. However, we include $w^k_1$ for the sake of completeness. Second, when $T=2$ and $b_k=1$, this algorithm reduces to the NASA algorithm presented in \cite{GhaRuswan20}. Indeed, \cref{alg:originalnasa} is a direct generalization of the NASA method to the multi-level case $T \ge 3$. However, to prove convergence of \cref{alg:originalnasa}, we need to take a batch of samples in each iteration to reduce the noise associated with estimation of the inner function values, when $T >2$. We now provide our convergence analysis for~\cref{alg:originalnasa}. To do so, we define the following filtration, $$\mathscr{F}_k:=\sigma(\{x^0,\ldots, x^k, z^0,\ldots, z^k, w_1^0,\ldots,w_1^k, \dots, w_T^0, \dots, w_T^k, u^0,\ldots, u^k\}).$$
Next, we state our main assumptions on the individual functions and the stochastic first-order oracle we use.

\begin{assumption}\label{fi_lips}
All functions $f_1,\dots,f_T$ and their derivatives are Lipschitz continuous with Lipschitz constants $L_{f_i}$ and $L_{\nabla f_i}$, respectively.
\end{assumption}
\begin{assumption}\label{assumption:original_nasa_assumption}
Denote $w_{T+1}^k \equiv x^k$. For each $k$, $w_{i+1}^k$ being the input, the stochastic oracle outputs $G_{i}^{k+1} \in \mathbb{R}^{d_i}$ and $J_{i}^{k+1} \in \mathbb{R}^{d_i \times d_{i-1}}$ such that
\begin{enumerate}
\item For $i \in \{1,\ldots,T\}$, we have $\mathbb{E}[J_{i}^{k+1}| \mathscr{F}_k] = \nabla f_i(w_{i+1}^k),~\text{and}~~\mathbb{E}[G_{i}^{k+1}| \mathscr{F}_k] = f_i(w_{i+1}^k)$.
\item For $i \in \{1,\ldots,T\}$, we have $\mathbb{E}[\| G_i^{k+1} - f_i(w_{i+1}^k) \|^2 | \mathscr{F}_k] \leq \sigma_{G_i}^2,\mathbb{E}[\| J_i^{k+1} - \nabla f_i(w_{i+1}^k) \|^2 | \mathscr{F}_k] \leq \hat \sigma_{J_i}^2,~\text{and}~~\mathbb{E}[\|J_i^{k+1}\|^2 | \mathscr{F}_k] \leq \sigma_{J_i}^2$. Here $\|\cdot\|$ denotes the Euclidean norm for vectors and the Frobenius norm for matrices.
\item Given $\mathscr{F}_k$, the outputs of the stochastic oracle at each level $i$, $G_{i}^{k+1}$ and $J_{i}^{k+1}$, are independent.
\item Given $\mathscr{F}_k$, the outputs of the stochastic oracle are independent between levels i.e., $\{G_{i}^{k+1}\}_{i=1,\ldots,T}$ are independent and so are $\{J_{i}^{k+1}\}_{i=1,\ldots,T}$.

\end{enumerate}

\end{assumption}
\cref{fi_lips} is a standard smoothness assumption made in the literature on nonlinear optimization. Similarly, Parts 1 and 2 in~\cref{assumption:original_nasa_assumption} are standard unbiasedness and bounded variance assumptions on the stochastic gradient,  common in the literature. At this point, we re-emphasize that the assumptions made in~\cite{zhang2019multi} are stronger than our assumptions above, as they require mean-square smoothness of the individual random functions $G_i$ and their gradients. Parts 3 and 4 are also essential to establish the convergence results in the multi-level case; similar assumptions have been made, for example, in~\cite{yang2019multi-level}. In the next couple of technical results, we provide some properties of composite functions that are required for our subsequent results.


\begin{lemma} \label{F_Lipschitz_Constant}
Define $F_i(x) = f_i \circ f_{i+1} \circ \cdots f_T(x)$. Under~\cref{fi_lips}, the gradient of $F_i$ is Lipschitz continuous with constant
\[
L_{\nabla F_i} = \sum_{j=i}^T \left[L_{\nabla f_j} \prod_{l=i}^{j-1} L_{f_l} \prod_{l=j+1}^T L_{f_l}^2\right].
\]

\begin{proof}
We show the result by backward induction. Under~\cref{fi_lips}, gradient of $F_T=f_T$ is Lipschitz continuous and so is that of $F_{T-1}$ since for any $x, y \in X$, we have
\begin{align*}
    \| \nabla F_{T-1}(x) - \nabla F_{T-1}(y)\|  &= \|\nabla f_T(x)\nabla f_{T-1}(f_T(x)) - \nabla f_T(y)\nabla f_{T-1}(f_T(y))\| \\ 
    & \leq \| \nabla f_T(x)\| \| \nabla f_{T-1}(f_T(x)) - \nabla f_{T-1}(f_T(y))\| \\
    &+ \| \nabla f_{T-1}(f_T(y))\| \| \nabla f_T(x) - \nabla f_T(y)\| \\ & \leq (L_{f_T}^2L_{\nabla f_{T-1}} + L_{f_{T-1}}L_{\nabla f_T})\|x -y\|.
\end{align*}
Now, suppose that gradient of $F_{i+1}$ is Lipschitz continuous for any $i \le T-1$. Then, similar to the above relation, $\nabla F_i$ is Lipschitz continuous with constant
\begin{align*}
L_{\nabla F_i} =& L_{F_{i+1}}^2 L_{\nabla f_i} + L_{f_i}L_{\nabla F_{i+1}}\\
= &L_{\nabla f_i} \prod_{j=i+1}^TL_{f_j}^2  + L_{f_i}\sum_{j=i+1}^T \left[L_{\nabla f_j} \prod_{l=i+1}^{j-1} L_{f_l} \prod_{l=j+1}^T L_{f_l}^2\right] \\
=&\sum_{j=i}^T \left[L_{\nabla f_j} \prod_{l=i}^{j-1} L_{f_l} \prod_{l=j+1}^T L_{f_l}^2\right].
\end{align*}
\end{proof}
\end{lemma}

We remark that the above result has also been proved in~\cite{zhang2019multi}, Lemma 5.2., with a slightly different proof.

\begin{lemma} \label{new_nasa_grad_F_z_hat_error}
Define $F_i(x) = f_i \circ f_{i+1} \circ \cdots f_T(x)$ and the gradient term \\
$\nabla \bar f_i(x,\bar w_i) = \nabla f_T(x) \nabla f_{T-1}(w_T)\cdots\nabla f_i(w_{i+1})$ with $\bar w_i = (w_{i+1},\ldots, w_T)$ for any $x \in X, w_j \in \bbr^{d_j} \ \  j=i+1,\ldots, T$. Then under~\cref{fi_lips}, we have
\begin{align*}
    \|\nabla F_i(x) - \nabla \bar f_i(x,\bar w_i)\| \leq \sum_{j=i}^{T-1}\frac{L_{\nabla f_j}}{L_{f_j}}L_{f_i} \cdots L_{f_T}\|F_{j+1}(x) - w_{j+1} \|.
\end{align*}
\end{lemma}
\begin{proof}
We show the result by backward induction. The case $i=T$ is trivial. When $i=T-1$, under~\cref{fi_lips}, we have
\begin{align*}
    \| \nabla F_{T-1}(x) - \nabla f_T(x)\nabla f_{T-1}(w_T)\| & = \| \nabla f_T(x) [\nabla f_{T-1}(f_T(x))- \nabla f_{T-1}(w_T)]\| \\
    &\leq L_{\nabla f_{T-1}}L_{f_T}\|f_T(x) - w_T\|.
\end{align*}
Now assume that for any $i \le T-2$,
\begin{align*}
    \|\nabla F_{i+1}(x) - \nabla \bar f_{i+1}(x,\bar w_{i+1} ) \| \leq \sum_{j=i+1}^{T-1}\frac{L_{\nabla f_j}}{L_{f_j}}L_{f_{i+1}} \cdots L_{f_T}\| F_{j+1}(x) - w_{j+1} \|.
\end{align*}
We then have
\begin{align*}
&\|\nabla F_i(x) - \nabla \bar f_i(x,\bar w_i ) \|
    = \|\nabla F_{i+1}(x)\nabla f_i(F_{i+1}(x)) - \nabla \bar f_i(x, \bar w_i )\|\\
&\le \|\nabla f_i(F_{i+1}(x))\| \|\nabla F_{i+1}(x) - \nabla \bar f_{i+1}(x, \bar w_{i+1} )\|\\
&+\|\nabla \bar f_{i+1}(x, \bar w_{i+1})\|\|\nabla f_i(F_{i+1}(x))-\nabla f_i(w_{i+1})\|
\\
&\le L_{f_i} \|\nabla F_{i+1}(x) - \nabla \bar f_{i+1}(x, \bar w_{i+1})\|+L_{\nabla f_i} L_{f_{i+1}} \cdots L_{f_T} \|F_{i+1}(x)-w_{i+1}\|\\
&\le L_{f_i} \sum_{j=i+1}^{T-1}\frac{L_{\nabla f_j}}{L_{f_j}}L_{f_{i+1}} \cdots L_{f_T}\| F_{j+1}(x) - w_{j+1} \|\\
&+L_{\nabla f_i} L_{f_{i+1}} \cdots L_{f_T} \|F_{i+1}(x)-w_{i+1}\|=\sum_{j=i}^{T-1}\frac{L_{\nabla f_j}}{L_{f_j}}L_{f_i} \cdots L_{f_T}\| F_{j+1}(x) - w_{j+1} \|.
\end{align*}
%
%
\end{proof}

\begin{lemma} \label{new_nasa_fjf_T_wj_error}
Under~\cref{fi_lips}, for any $ j \in \{1,\ldots, T-1\}$, we have
\begin{align*}
    \|f_j \circ \cdots \circ f_T(x) - w_j\| & \leq \|f_j(w_{j+1}) - w_j\| + \sum_{\ell = j+1}^T\left(\prod_{i = j}^{\ell - 1}L_{f_i}\right)\|f_{\ell}(w_{\ell+1}) - w_{\ell}\|.
\end{align*}
\end{lemma}
\begin{proof}
We show the results by backward induction. For $j = T-1$, we have
\begin{align*}
    &\|f_{T-1} \circ f_T(w_{T+1}) - w_{T-1} \| \\
    &\leq \| f_{T-1} \circ f_T(w_{T+1}) - f_{T-1}(w_T) \| + \|f_{T-1}(w_T) - w_{T-1}\| \\ & \leq L_{f_{T-1}}\|f_T(w_{T+1}) - w_T\| + \| f_{T-1}(w_T) - w_{T-1}\|.
\end{align*}
Now suppose the result holds for $j+1, j \in \{1,\ldots, T-2\}$. Then, we have
\begin{align*}
  &  \| f_j \circ f_{j+1} \circ \cdots f_T(w_{T+1}) - w_j\| \\
   \leq & \|f_{j} \circ \cdots f_T(w_{T+1}) - f_j(w_{j+1}) + f_j(w_{j+1}) - w_j\| \\
    \leq & L_{f_j}\|f_{j+1} \circ \cdots \circ f_T(w_{T+1}) - w_{j+1}\| + \| f_j(w_{j+1}) - w_j\| \\ 
     \leq & L_{f_j}\left[\| f_{j+1}(w_{j+2}) - w_{j+1}\| + \sum_{\ell = j+2}^T\left(\prod_{i=j+1}^{\ell -1}L_{f_i}\right)\|f_{\ell}(w_{\ell+1}) - w_{\ell}\|\right] \\ & + \|f_j(w_{j+1}) - w_j\| \\  = & \|f_j(w_{j+1}) - w_j\| + \sum_{\ell = j+1}^T\left(\prod_{i=j}^{\ell-1}L_{f_i}\right)\|f_{\ell}(w_{\ell+1}) - w_{\ell}\|,
\end{align*}
where the third inequality follows by the induction hypothesis.
\end{proof}

\begin{lemma} \label{FZ_difference}
Define
\begin{align*}
    R_1 & = L_{\nabla f_1}L_{f_2} \cdots L_{f_T}, \qquad R_j = L_{f_1} \dots L_{f_{j-1}}L_{\nabla f_j}L_{f_{j+1}} \cdots L_{f_T}/L_{f_j}  \quad 2 \le j \leq T-1, \\ 
    C_2  &= R_1, \quad C_j  = \sum_{i=1}^{j-2} R_i \left(\prod_{l=i+1}^{j-1}L_{f_l}\right) \quad 3 \le j \leq T.
\end{align*}
Assume that~\cref{fi_lips} holds. Then for $T \geq 3$,
\beq
    \left\| \nabla F(x) - \nabla f_T(x) \prod_{i=2}^T \nabla f_{T+1-i}(w_{T+2-i}) \right\| \leq  \sum_{j=2}^{T-1} C_j\|f_j(w_{j+1}) - w_j\| + C_T\|f_T(x) - w_T\|
\eeq
\end{lemma}
\begin{proof}
By~\cref{new_nasa_grad_F_z_hat_error} and \cref{new_nasa_fjf_T_wj_error}, we have
\begin{align*}
    &  \left\| \nabla F(x) - \nabla f_T(x) \prod_{i=2}^T \nabla f_{T+1-i}(w_{T+2-i}) \right\| \leq \sum_{j=1}^{T-1}R_j\|f_{j+1} \circ \cdots \circ f_T(x) - w_{j+1}\| \\ & = \sum_{j=1}^{T-2}R_j\|f_{j+1} \circ \cdots \circ f_T(x) - w_{j+1}\| + R_{T-1}\|f_T(x) - w_T\| \\ & = \sum_{j=1}^{T-2}R_j\|f_{j+1}(w_{j+2}) - w_{j+1}\| + \sum_{j=1}^{T-2}R_j\sum_{\ell = j+2}^T\left(\prod_{i=j+1}^{\ell-1}L_{f_i}\right)\|f_{\ell}(w_{\ell +1}) -w_{\ell}\| \\
    &+ R_{T-1}\|f_T(x) - w_T\|.
\end{align*}
{\color{black}
Aggregating the constants for $\|f_j(w_{j+1})-w_j\|$, we get the result.}
\end{proof}

\vgap







\textcolor{black}{The following result also shows the Lipschitz continuity of the gradient of the objective function} of the subproblem \eqnok{def_uk}. One can see \cite{GhaRuswan20} for a simple proof.
\begin{lemma}\label{eta_lips_lemma}
Let $\eta(x,z)$ be defined as
\[
\eta(x,z) = \min_{y \in X}\  \left\{\langle z, y-x \rangle + \frac{\beta}{2} \|y-x\|^2\right\}.
\]
Then the gradient of $\eta$ w.r.t. $(x,z)$ is Lipschitz continuous with the constant \[L_{\nabla \eta} = 2\sqrt{(1+\beta)^2+(1+\tfrac{1}{2\beta})^2}.\]
\end{lemma}

\vgap
\noindent In the next result, we provide a recursion inequality for the error in estimating $f_i(w_{i+1})$ by $w_i$.

\begin{lemma} \label{nasa_lemma_fn_w_error}
Let $\{x^k\}_{k \ge 0}$ and $\{w^k_i\}_{k \ge 0} \ \ 1 \le i \le T$ be generated by~\cref{alg:originalnasa}. Denote
\beq \label{nasa_A_ki}
d^k = u^k - x^k, \qquad w_{T+1}^k \equiv x^k \quad \forall k \ge 0, \qquad A_{k,i} = f_i(w_{i+1}^{k+1}) - f_i(w_{i+1}^k) \quad 1 \le i \le T.
\eeq
\begin{itemize}[leftmargin=0.1in]
\item [a)] For any $i \in \{1,\ldots, T\}$,
\begin{align}
   & \| f_i(w_{i+1}^{k+1}) - w_i^{k+1}\|^2 \leq (1-\tau_k)\| f_i(w_{i+1}^k) - w_i^k\|^2 + \frac{1}{\tau_k}\|A_{k,i}\|^2 + \tau_k^2\|e_i^{k+1}\|^2 
    + r_i^{k+1},\label{nasa_fn_w}\\
   &\|w_i^{k+1} - w_i^k\|^2  \leq \tau_k^2\left[\| f_i(w_{i+1}^k) - w_i^k\|^2 + \|e_i^{k+1}\|^2 - 2\langle e_i^{k+1},f_i(w_{i+1}^k) - w_i^k\rangle \right],\label{nasa_wi_increment_squared}
\end{align}
where
\beq\label{nasa_ri}
    r_i^{k+1} = 2\tau_k \langle e_i^{k+1}, A_{k,i} + (1-\tau_k)(f_i(w_{i+1}^k) - w_i^k) \rangle, \qquad e_i^{k+1}= f_i(w_{i+1}^k)-\bar G_i^{k+1}.
\eeq
\item [b)] If, in addition, $f_i$'s are Lipschitz continuous, we have
\begin{align}
\| f_T(x^{k+1}) - w_T^{k+1}\|^2 &\leq (1-\tau_k)\| f_T(x^k) - w_T^k\|^2 + L_{f_T} \tau_k \|d^k\|^2 + \tau_k^2\|e_T^{k+1}\|^2 + r_T^{k+1}, \\
    \| f_i(w_{i+1}^{k+1}) - w_i^{k+1}\|^2 & \leq (1-\tau_k)\|f_i(w_{i+1}^k) - w_i^k\|^2 +   \tau_k^2\|e_i^{k+1}\|^2 + \bar{r}_i^{k+1} \notag \\ & 
 + L_{f_i}^2\tau_k \left[\|f_{i+1}(w_{i+2}^k) - w_{i+1}^k\|^2 + \|e_{i+1}^{k+1}\|^2\right]  \qquad   1 \le i \le T-1, \label{nasa_upperbound_fiwi_squared}
\end{align}
where
\begin{align}
    \bar{r}_{i}^{k+1} & = -2\tau_kL_{f_i}^2\langle e_{i+1}^{k+1},f_{i+1}(w_{i+2}^k)-w_{i+1}^k\rangle + r_i^{k+1}. \label{original_nasa_ri_bar}
\end{align}
\end{itemize}
\end{lemma}

\begin{proof}
Noting \cref{def_wi}, \cref{nasa_fn_w}, and \cref{nasa_ri}, we have
\begin{align*}
    \|f_i(w_{i+1}^{k+1}) - w_i^{k+1}\|^2 & = \|A_{k,i} + f_i(w_{i+1}^k) - (1-\tau_k)w_i^k - \tau_k(f_i(w_{i+1}^k) - e_i^{k+1})\|^2 \\ & = \|A_{k,i} + (1-\tau_k)(f_i(w_{i+1}^k) - w_i^k) + \tau_ke_i^{k+1}\|^2 \\ & = \|A_{k,i} +(1-\tau_k)(f_i(w_{i+1}^k)-w_i^k)\|^2 + \tau_k^2\|e_i^{k+1}\|^2 + r_i^{k+1}.
\end{align*}
Then, in the view of \cref{nasa_ri}, \cref{nasa_fn_w} follows by noting that
\begin{align*}
   & \|A_{k,i} +(1-\tau_k)(f_i(w_{i+1}^k)-w_i^k)\|^2 \\
    = & \|A_{k,i}\|^2 + (1-\tau_k)^2\|f_i(w_{i+1}^k)-w_i^k\|^2 + 2(1-\tau_k)\langle A_{k,i}, f_i(w_{i+1}^k) - w_i^k\rangle \nonumber\\
     \leq & \|A_{k,i}\|^2 + (1-\tau_k)^2\|f_i(w_{i+1}^k) - w_i^k\|^2 + \left(\frac{1}{\tau_k}-1\right)\|A_{k,i}\|^2 \nonumber\\ & + (1-\tau_k)\tau_k\|f_i(w_{i+1}^k)-w_i^k\|^2 \nonumber\\
     = & (1-\tau_k)\|f_i(w_{i+1}^k)-w_i^k\|^2 + \frac{1}{\tau_k}\|A_{k,i}\|^2,\label{Ak_fk}
\end{align*}
due to Cauchy-Schwarz and Young's inequalities. Also, \cref{nasa_wi_increment_squared} directly follows from \cref{def_wi} since
\begin{align*}
    \|w_i^{k+1} - w_i^k\|^2 & = \|\tau_k(\bar G_i^{k+1} - w_i^k)\|^2 = \tau_k^2\|f_i(w_{i+1}^k) - w_i^k - e_i^{k+1}\|^2 \\
    & = \tau_k^2\left[\| f_i(w_{i+1}^k) - w_i^k\|^2 + \|e_i^{k+1}\|^2 - 2\langle e_i^{k+1},f_i(w_{i+1}^k) - w_i^k\rangle \right].
\end{align*}
To show part b), note that by \cref{def_xk}, \cref{nasa_A_ki}, and Lipschitz continuity of $f_i$, we have
\[
\| A_{k,T} \| \leq L_{f_T}\| w_{T+1}^{k+1} - w_{T+1}^k\|=L_{f_T} \tau_k \|d^k\|, \qquad \| A_{k,i} \| \leq L_{f_i}\| w_{i+1}^{k+1} - w_{i+1}^k\|,
\]
for $1 \le i \le T-1$. The results then follows by noting \cref{nasa_fn_w} and \cref{nasa_wi_increment_squared}.
\end{proof}


We remark that the mini-batch sampling in \eqnok{def_barG} is only used to reduce the upper bound on the expectation of
$\tau_k \|e_{i+1}^{k+1}\|^2$ in the right hand side of \eqnok{nasa_upperbound_fiwi_squared}. Moreover, we do not need this inequality for $i=1$ when establishing the convergence rate of \cref{alg:originalnasa}. Thus, when $T \le 2$, this algorithm converges without using mini-batch of samples in each iteration, as shown in \cite{GhaRuswan20}.

\textcolor{black}{Recalling the definition of $F^*$ from Section~\ref{sec:intro} and denoting $w:= (w_1,\dots,w_T)$, we define, for some positive constants $\gamma =(\gamma_1,\ldots, \gamma_T)$, the merit function
\beq\label{merit_function}
    W_\gamma(x,z,w) = F(x) - F^* - \eta(x,z) + \sum_{i=1}^{T-1}\gamma_i \| f_i(w_{i+1}) - w_i\|^2 + \gamma_T \| f_T(x) - w_T\|^2,
\eeq
which will be used in our next result for establishing convergence analysis of \cref{alg:originalnasa}. It is worth noting that  $W_\gamma(x,z,w) \ge 0$ due to that facts that $F(x) \ge F^*, \eta(x,z) \le0$ (by Lemma~\ref{eta_lips_lemma}), and $\gamma >0$. The precise values of the constants $\gamma_1,\ldots, \gamma_T$ will be set later in our analysis. We should also mention that the above summation can start from $i=2$, in which case the convergence analysis is slightly simpler. However, we use~\eqnok{merit_function} in our analysis since, as a byproduct, it gives us an online certificate for the stochastic values of the objective function.} The above function is an extension of the one used in \cite{GhaRuswan20} for the case of $T=2$, to the multi-level setting of $T\ge 1$. A variant of this function (including only the first two terms in \eqnok{merit_function}) was used in the literature as early as \cite{ruszczynski1987linearization} and was used later in \cite{rusz20-1} for nonsmooth single-level stochastic optimization. 


\begin{lemma} \label{original_nasa_merit_function_lemma}
Suppose that the sequences $\{x^k,z^k,u^k,w_1^k,\dots,w_T^k\}_{k \geq 0}$ are generated by~\cref{alg:originalnasa} and~\cref{fi_lips} holds.
\begin{itemize}[leftmargin=0.1in]
\item [a)] If
\begin{align}\label{parameter_cond}
&\gamma_1 \ge \lambda >0, \qquad \gamma_j-\gamma_{j-1}L^2_{f_{j-1}}-\lambda >0, \notag\\ 
&4(\beta-\lambda-\gamma_T)(\gamma_j-\gamma_{j-1}L^2_{f_{j-1}}-\lambda) \ge TC_j^2 \quad j=2,\ldots, T,
\end{align}
where $C_j$'s are defined in~\cref{FZ_difference}, we have
\begin{align}\label{main_recursion}
\lambda \sum_{k=0}^{N-1}\tau_k\left[\|d^k\|^2 + \sum_{i=1}^{T-1}\|f_i(w_{i+1}^k) - w_i^k\|^2 + \|f_T(x^k) - w_T^k\|^2 \right] &\leq W_\gamma(x^0,z^0,w^0)\notag\\
&+\sum_{k=0}^{N-1}R^{k+1},
\end{align}
where
\begin{align}
R^{k+1} &:= \tau_k^2 \sum_{i=1}^{T}\gamma_i \| e_i^{k+1}\|^2+\tau_k \sum_{i=1}^{T-1}\gamma_i L_{f_i}^2 \| e_{i+1}^{k+1}\|^2 + \sum_{i=1}^{T-1}\gamma_i \bar{r_i}^{k+1}+\gamma_T r_T^{k+1} \notag\\
&~~~~~+
\tau_k \langle d^k, \Delta^k\rangle+\frac{(L_{\nabla F}+L_{\nabla \eta})\tau_k^2}{2}\|d^k\|^2 + \frac{L_{\nabla \eta}}{2}\|z^{k+1} - z^k\|^2, \label{def_R_error}\\
 \Delta^k&:= \nabla f_T(x^k) \prod_{i=2}^T \nabla f_{T+1-i}(w^k_{T+2-i})-\prod_{i=1}^T J_{T-i+1}^{k+1},\label{def_deltak}
\end{align}
and $r^{k+1}_i, \bar r^{k+1}_i$ are defined in \cref{nasa_ri} and \cref{original_nasa_ri_bar}, respectively.
\item[b)] If parameters are chosen as
\begin{align}\label{parameter_def}
\gamma_j &:= 2^{j-1}(L_{f_1}\cdots L_{f_{j-1}})^2 \sqrt{T} \quad 2 \leq j \leq T, \qquad \beta \geq \lambda + \gamma_T+\frac{T \max_{2 \leq i \leq T}C_i^2}{4\lambda},\notag\\
\gamma_1 &= \lambda = \frac{1}{2}\min_{2 \leq i \leq T}(\gamma_i - \gamma_{i-1}L_{f_{i-1}}^2)= \frac{\min_{2 \leq i \leq T} \gamma_i}{4}.
\end{align}
Then, conditions in \cref{parameter_cond} are satisfied.
\end{itemize}
%
%
\end{lemma}

\begin{proof}
First, note that by~\cref{F_Lipschitz_Constant}, we have
\begin{align}\label{F_lips}
F(x^{k+1}) &\le F(x^k) +\langle \nabla F(x^k), x^{k+1}-x^k \rangle +\frac{L_{\nabla F}}{2}\|x^{k+1}-x^k\|^2 \notag \\
&= F(x^k) +\tau_k \langle \nabla F(x^k), d^k \rangle + \frac{L_{\nabla F}\tau_k^2}{2}\|d^k\|^2.
\end{align}
Second, note that by the optimality condition of \cref{def_uk}, we have
\beq \label{opt_QP}
\langle z^k+\beta(u^k-x^k), x^k-u^k \rangle \ge 0,~\text{which implies}~\langle z^k, d^k \rangle +\beta \|d^k\|^2 \le 0.
\eeq
Then, noting \cref{def_xk}, \cref{def_zk}, and in the view of~\cref{eta_lips_lemma}, we obtain
\begin{align}\label{eta_lips}
&\eta(x^k,z^k)- \eta(x^{k+1},z^{k+1}) \notag\\
&\le \langle z^k+ \beta(u^k-x^k), x^{k+1}-x^k \rangle- \langle u^k-x^k, z^{k+1}-z^k\rangle \notag\\ 
&+\frac{L_{\nabla \eta}}{2} \left[\|x^{k+1}-x^k\|^2 +\|z^{k+1}-z^k\|^2 \right]\notag \\
&= \tau_k \langle 2z^k+ \beta d^k, d^k \rangle- \tau_k \langle d^k, \prod_{i=1}^TJ_{T-i+1}^{k+1}\rangle +\frac{L_{\nabla \eta}}{2} \left[\|x^{k+1}-x^k\|^2 +\|z^{k+1}-z^k\|^2 \right] \notag \\
&\leq -\beta\tau_k \|d^k\|^2 - \tau_k\langle d^k, \prod_{i=1}^TJ_{T-i+1}^{k+1}\rangle
    + \frac{L_{\nabla \eta}}{2}\left[\tau_k^2\|d^k\|^2 + \|z^{k+1} - z^k\|^2 \right].
\end{align}
Third, noting~\cref{nasa_lemma_fn_w_error}.b), we have
\begin{align}
&\sum_{i=1}^{T-1}\gamma_i \left[\| f_i(w^{k+1}_{i+1}) - w^{k+1}_i\|^2 - \| f_i(w^k_{i+1}) - w^k_i\|^2 \right] \notag\\ 
&\qquad+ \gamma_T \left[\| f_T(x^{k+1}) - w^{k+1}_T\|^2-\| f_T(x^k) - w^k_T\|^2\right]\notag\\
\le& \sum_{i=1}^{T-1}\gamma_i\bigg\{-\tau_k\big[\|f_i(w_{i+1}^k) - w_i^k\|^2- L_{f_i}^2\|f_{i+1}(w_{i+2}^k) - w_{i+1}^k\|^2 \notag \\
&\qquad - L_{f_i}^2 \| e_{i+1}^{k+1}\|^2 \big]+ \tau_k^2 \|e_i^{k+1}\|^2 + \bar{r_i}^{k+1}\bigg\} \notag\\
&\qquad + \gamma_T \left\{-\tau_k\left[\|f_T(x^k) - w_T^k\|^2 - L_{f_T}^2\|d^k\|^2 \right]+\tau_k^2 \|e_T^{k+1}\|^2 + r_T^{k+1}\right\}\notag\\
=&- \tau_k \bigg\{\gamma_1 \|f_1(w_2^k) - w_1^k\|^2  + \sum_{j=2}^{T-1}[\gamma_j-\gamma_{j-1}L_{f_{j-1}}^2] \|f_j(w_{j+1}^k) - w_j^k\|^2 \notag \\
&\qquad + [\gamma_T-\gamma_{T-1}L_{f_{T-1}}^2]\|f_T(x^k) - w_T^k\|^2\bigg\}+ \sum_{i=1}^{T-1}\gamma_i \bar{r_i}^{k+1}+\gamma_T r_T^{k+1}\notag\\
&\qquad+\tau_k \left[\sum_{i=1}^{T-1}\gamma_i L_{f_i}^2 \| e_{i+1}^{k+1}\|^2+\gamma_T\|d^k\|^2\right]
+\tau_k^2 \sum_{i=1}^{T}\gamma_i \| e_i^{k+1}\|^2.\label{fun_diff}
\end{align}
Combining the above relation with \cref{eta_lips}, \cref{F_lips}, noting definition of merit function in \cref{merit_function}, and in the view of~\cref{FZ_difference}, we obtain
\begin{align*}
    & W_\gamma(x^{k+1},z^{k+1},w^{k+1}) - W_\gamma(x^k,z^k,w^k)   \\
    &\le -\tau_k(\beta-\gamma_T)\|d^k\|^2+\tau_k \|d^k\|\left[\sum_{j=2}^{T-1} C_j\|f_j(w^k_{j+1}) - w^k_j\| + C_T\|f_T(x) - w_T\|\right]\\
    &\qquad- \tau_k \bigg\{\gamma_1 \|f_1(w_2^k) - w_1^k\|^2  + \sum_{j=2}^{T-1}[\gamma_j-\gamma_{j-1}L_{f_{j-1}}^2] \|f_j(w_{j+1}^k) - w_j^k\|^2 \notag \\
    &\qquad+ [\gamma_T-\gamma_{T-1}L_{f_{T-1}}^2]\|f_T(x^k) - w_T^k\|^2\bigg\}+R^{k+1},\notag
\end{align*}
where $R^{k+1}$ is defined in \cref{def_R_error}. {\color{black}Now, if \cref{parameter_cond} holds, we have
\begin{align*}
    & -\left(\frac{\beta-\gamma_T}{T}\right)\|d^k\|^2- (\gamma_j-\gamma_{j-1}L_{f_{j-1}}^2) \|f_j(w_{j+1}^k) - w_j^k\|^2 + C_j \|d^k\| \|f_j(w^k_{j+1}) - w^k_j\| \\
    &\le -\lambda \Big[\frac{1}{T}\|d^k\|^2 +\|f_j(w_{j+1}^k) - w_j^k\|^2 \Big] \qquad \forall j \in \{1,\ldots,T\},
\end{align*}
which together with the above inequality imply that}
\begin{align}
&W_\gamma(x^{k+1},z^{k+1},w^{k+1}) - W_\gamma(x^k,z^k,w^k) \notag\\
&\le -\lambda \tau_k\left[\|d^k\|^2 + \sum_{i=1}^{T-1}\|f_i(w_{i+1}^k) - w_i^k\|^2 + \|f_T(x^k) - w_T^k\|^2 \right]+R^{k+1}.\notag
\end{align}
Summing up the above inequalities and re-arranging the terms, we obtain \cref{main_recursion}. It can be easily verified that condition \cref{parameter_cond} is satisfied by the choice of parameters in \cref{parameter_def}.
\end{proof}

The next technical result helps us to simplify our convergence analysis.

\begin{lemma} \label{Gamma_lemma} Consider a sequence $\{\tau_k\}_{k \geq 0} \in (0,1]$, and define
\begin{align}\label{def_Gamma}
\Gamma_k= \Gamma_1 \prod_{i=1}^{k-1}(1-\tau_i) \qquad k \ge 2, \qquad
\Gamma_1 = \begin{cases}1 & \text{ if } \tau_0 = 1, \\ 1-\tau_0 & \text{otherwise.}\\\end{cases}
\end{align}
\begin{itemize}
\item [a)] For any $k \ge 1$, we have
\[
\alpha_{i,k} = \frac{\tau_i}{\Gamma_{i+1}}\Gamma_k \quad 1 \le i \le k, \qquad \sum_{i=0}^{k-1}\alpha_{i,k}=\begin{cases}1 & \text{ if } \tau_0 = 1, \\ 1-\Gamma_k & \text{otherwise.}\\\end{cases}
\]

\item [b)]Suppose that $q_{k+1} \leq (1-\tau_k)q_k + p_k \ \  k \ge 0$ for sequences $\{q_k,p_k\}_{k \geq 0}$. Then, we have
\[
    q_k \leq \Gamma_k \left[a q_0 + \sum_{i=0}^{k-1}\frac{p_i}{\Gamma_{i+1}}\right], \qquad a = \begin{cases}0 & \text{ if } \tau_0 = 1, \\ 1 & \text{otherwise.}\\\end{cases}
\]
\end{itemize}
\end{lemma}

\begin{proof}
To show part a), note that
\begin{align*}
\sum_{i=0}^{k-1} \alpha_{i,k}&=\Gamma_k \sum_{i=0}^{k-1}\frac{\tau_i}{\Gamma_{i+1}}= \frac{\tau_0 \Gamma_k}{\Gamma_1} + \sum_{i=1}^{k-1}\frac{\tau_i\Gamma_k}{\Gamma_{i+1}} = \frac{\tau_0 \Gamma_k}{\Gamma_1} + \Gamma_k\sum_{i=1}^{k-1}\left(\frac{1}{\Gamma_{i+1}}-\frac{1}{\Gamma_i}\right) \\
&= 1 - \frac{\Gamma_k}{\Gamma_1}(1-\tau_0).
\end{align*}
To show part b), by dividing both sides of the inequality by $\Gamma_{k+1}$ and  noting \cref{def_Gamma}, we have
\[
\frac{q_1}{\Gamma_1} \leq \frac{(1-\tau_0)q_0+p_0}{\Gamma_1},\qquad
\frac{q_{k+1}}{\Gamma_{k+1}} \leq \frac{q_k}{\Gamma_k} + \frac{p_k}{\Gamma_{k+1}} \quad k \ge 1.
\]
Summing up the above inequalities, we get the result.
\end{proof}

\vgap

The next result shows the boundedness of the error terms in the right hand side of \eqnok{main_recursion} in expectation. This is an essential step in establishing the convergence analysis of the algorithm.

\begin{proposition} \label{original_nasa_big_proposition}
Suppose that~\cref{assumption:original_nasa_assumption} holds and (for simplicity) $\tau_0 = 1$, $ \beta > 0$ for all $k$. Then, for any  $k \geq 1$, we have
\begin{align}
    \beta^2 \mathbb{E}[\|d^k\|^2| \mathscr{F}_{k}] &\leq \mathbb{E}[\|z^k\|^2|\mathscr{F}_{k}] \leq \prod_{i=1}^T\sigma_{J_i}^2, \label{dk_bnd}\\
    \mathbb{E}[\|z^{k+1}-z^k\|^2 | \mathscr{F}_k] & \leq 4\tau_k^2\prod_{i=1}^T\sigma_{J_i}^2. \label{zk_error}
\end{align}
If, in addition, the batch size $b_k$ in~\cref{alg:originalnasa} is set to
\beq\label{def_bk}
b_k = \left\lceil \frac{\max_{1 \le i \le T} L^2_{f_i}}{\tau_k} \right\rceil \qquad k \ge 0,
\eeq
we have
\beq\label{Rk_bnd}
\mathbb{E}[R^{k+1} | \mathscr{F}_k] \leq \tau_k^2 \left[\frac{1}{2}\left(\prod_{i=1}^T\sigma_{J_i}^2\right)\left(\frac{L_{\nabla F}+(1+4\beta^2)L_{\nabla \eta}}{\beta^2} \right) + \sum_{i=1}^T\gamma_i \sigma_{G_i}^2\right]:= \tau_k^2 \sigma^2,
\eeq
where $R^{k+1}$ is defined in \cref{def_R_error}.
%
%

\end{proposition}

\begin{proof}
The first inequality in \cref{dk_bnd} directly follows by \cref{opt_QP} and Cauchy-Schwarz inequality. Noting \cref{def_zk}, the fact that $\tau_0 = 1$, and in the view of~\cref{Gamma_lemma}, we obtain
%
%
%

\begin{align*}
    z^k & = \sum_{i=0}^{k-1}\alpha_{i,k}\left(\prod_{\ell = 1}^T J_{T+1-l}^{i+1}\right)
\end{align*}

By convexity of $\| \cdot \|^2$ and conditional independence, we conclude that

\begin{align*}
    \mathbb{E}[\|z^k\|^2 | \mathscr{F}_k] &\leq \sum_{i=0}^{k-1}\alpha_{i,k}\mathbb{E}\left[\left\| \prod_{\ell=1}^TJ_{\ell}^{i+1} \right\|^2 \ \Bigg| \ \mathscr{F}_k\right] & \leq \sum_{i=0}^{k-1}\alpha_{i,k}\prod_{\ell = 1}^T \mathbb{E}[\| J_{\ell}^{i+1}\|^2 | \mathscr{F}_i] \\
    &\leq \sum_{i=0}^{k-1}\alpha_{i,k}\left(\prod_{\ell=1}^T\sigma_{J_{\ell}}^2 \right) = \prod_{\ell =1}^T\sigma_{J_{\ell}}^2.
\end{align*}
%
%
Noting \cref{dk_bnd}, we have
\begin{align*}
    \mathbb{E}[\| z^{k+1}-z^k\|^2 | \mathscr{F}_k] & \leq  \tau_k^2 \mathbb{E}\left[\left\| z^k - \prod_{\ell=1}^TJ_{\ell}^{k+1} \right\|^2 \ \Bigg| \ \mathscr{F}_k\right] \\
    &\leq 2\tau_k^2\left\{\mathbb{E}[\|z^k\|^2 | \mathscr{F}_k] + \mathbb{E}\left[\left\| \prod_{\ell=1}^T J_{\ell}^{k+1} \right\|^2 \ \Bigg| \mathscr{F}_k\right]\right\} \\
    & \leq 2\tau_k^2\left(\prod_{\ell=1}^T\sigma_{J_{\ell}}^2 + \prod_{\ell = 1}^T\sigma_{J_{\ell}}^2\right)  = 4\tau_k^2\left(\prod_{\ell=1}^T\sigma_{J_{\ell}}^2\right).
\end{align*}
Now, observe that by \cref{nasa_ri}, \cref{original_nasa_ri_bar}, the choice of $b_k$ in \cref{def_bk}, and under~\cref{assumption:original_nasa_assumption}, we have
\begin{align*}
\mathbb{E}[\Delta^k | \mathscr{F}_k] &=0, \qquad \mathbb{E}[e_i^{k+1} | \mathscr{F}_k] =0, \quad \text{implying} \quad \mathbb{E}[r_i^{k+1} | \mathscr{F}_k]=\mathbb{E}[\bar r_i^{k+1} | \mathscr{F}_k]=0,\\
\mathbb{E}[\| e_i^{k+1} \|^2 | \mathscr{F}_k] &=\mathbb{E}[\| \frac{1}{b_k}\sum_{j=1}^{b_k}G_{i,j}^{k+1} - f_i(w_{i+1}^k) \|^2 | \mathscr{F}_k]
\leq \frac{\sigma_{G_i}^2}{b_k}\\
&\qquad \qquad \qquad \qquad \qquad \qquad \qquad  \qquad  \le \min\left\{1, \frac{\tau_k }{\max_{1 \le i \le T} L^2_{f_i}}\right\} \sigma_{G_i}^2.
\end{align*}
Noting \cref{def_R_error}, \cref{dk_bnd}, \cref{zk_error}, and the above observation, we obtain \cref{Rk_bnd}.
\end{proof}

Observe that \cref{original_nasa_merit_function_lemma} shows that the summation of $\|d^k\|$ and the errors in estimating the inner function values are bounded by summation of error terms $R^k$ which is in the order of $\sum_{k=1}^N \tau_k^2$ as shown in \cref{original_nasa_big_proposition}. This is the main step in establishing the convergence of \cref{alg:originalnasa}. Indeed, $\bar{x} \in X$ is a stationary point of \cref{eq:mainprob}, if $\bar u=\bar{x}$ and $\bar{z}=\nabla F(\bar{x})$, where
\beq\label{def_sub2}
\bar u = \argmin_{y \in X}\  \left\{\langle \bar{z}, y-\bar{x} \rangle + \frac{1}{2} \|y-\bar{x}\|^2\right\}.
\eeq
Thus, for a given pair of $(\bar{x},\bar{z})$, we can define our termination criterion as follows.
\begin{definition}\label{def_Vxz}
A pair of $(\bar{x},\bar{z})$ generated by ~\cref{alg:originalnasa} is called an $\epsilon$-stationary pair, if $\mathbb{E}[\sqrt{V(\bar{x},\bar{z})}] \le \epsilon$, where
\begin{align}
    V(\bar x,\bar z) = \|\bar u-\bar x\|^2 + \|\bar z - \nabla F(\bar x)\|^2, \label{Lyaponuv_function}
\end{align}
and $\bar u$ is the solution to \eqnok{def_sub2}.
\end{definition}
\textcolor{black}{We emphasize that in Definition~\ref{def_Vxz}, we consider a unified termination criterion for both the unconstrained and constrained cases. When $X=\bbr^{d_T}$, $V(\bar x,\bar z)$ provides an upper bound for the $\|\nabla F(\bar x)\|^2$. This can be simply seen from the fact that $\bar u-\bar x = \bar z$ in~\eqref{def_sub2} for unconstrained problems and hence from~\eqref{Lyaponuv_function}, we have
\[
V(\bar x,\bar z) = \|\bar z\|^2+\|\bar z -\nabla F(\bar x)\|^2 \ge \frac 12 \|\nabla F(\bar x)\|^2.
\] 
We also refer the reader to \cite{GhaRuswan20} for the relation between $V(\bar{x},\bar{z})$ and other common gradient-based termination criteria used in the literature such as gradient mapping (\cite{Nest04,Nest07-1,GhaLanZhang16}) and proximal mapping (\cite{DruLew18}). Furthermore, as shown in \cite{GhaRuswan20}, we have
\begin{align}
    V(x^k,z^k) \le \max(1,\beta^2)\|u^k-x^k\|^2 + \|z^k - \nabla F(x^k)\|^2, \label{Lyaponuv_function2}
\end{align}
where $(x^k,u^k, z^k)$ are the solutions generated at iteration $k-1$ of \cref{alg:originalnasa}. Noting this fact, we provide the convergence rate of this algorithm by appropriately choosing $\beta$ and $\tau_k$ in the next results.} 

\begin{theorem}\label{thm:main_theorem_original_nasa}
Suppose that $\{x^k,z^k\}_{k \geq 0}$ are generated by~\cref{alg:originalnasa}, ~\cref{fi_lips} and \cref{assumption:original_nasa_assumption} hold. Also assume that the parameters satisfy \cref{parameter_def} and step sizes $\{\tau_k\}$ are chosen such that
\begin{align}
    \sum_{i=k+1}^N \tau_i\Gamma_i \leq c \Gamma_{k+1} \quad \forall k \geq 0 \text{ and } \forall N \geq 1, c \text{ is a positive constant}. \label{gamma_condition}
\end{align}

\textbf{(a)} For every $N \geq 1$, we have
\beq\label{nasa_main_theorem}
    \sum_{k=1}^N\tau_k \mathbb{E}[\| \nabla F(x^k) - z^k\|^2 | \mathscr{F}_{k}] \leq {\cal B}_1(\sigma^2, N),
\eeq
where
\beq\label{def_B}
{\cal B}_1(\sigma^2,N)=\frac{4cL^2(T-1)}{\lambda}\left[W_\gamma(x^0,z^0,w^0) + \sigma^2 \sum_{k=0}^{N-1} \tau_k^2\right] + c \prod_{\ell=1}^T\sigma_{J_{\ell}}^2 \sum_{k=0}^{N-1} \tau_k^2,
\eeq
$\sigma^2$ is defined in \cref{Rk_bnd} and
\beq\label{def_L}
    L^2 = \max \left\{L_{\nabla F}^2, \max_{2 \leq i \leq T}C_j^2\right\}.
\eeq

\textbf{(b)} As a consequence, we have
\begin{align}
    \mathbb{E}[V(x^R,z^R)] &\le  \frac{1}{\sum_{k=1}^{N} \tau_k}\left\{{\cal B}_1(\sigma^2,N)+\frac{\max(1,\beta^2)}{\lambda}\left[W_\gamma(x^0,z^0,w^0)+ \sigma^2 \sum_{k=0}^{N} \tau_k^2\right]\right\},
    \label{nasa_main_theorem2}
\end{align}
where the expectation is taken with respect to all random sequences generated by the method and an independent random integer number $R \in \{1,\dots, N\}$, whose probability distribution is given by
\begin{align*}
    \mathbb{P}[R = k] = \frac{\tau_k}{\sum_{j=1}^{N}\tau_j}
\end{align*}

\textbf{(c)} If, in addition, the stepsizes are set to

\beq\label{def_tau}
    \tau_0 = 1, \quad \tau_k = \frac{1}{\sqrt{N}} \quad \forall k = 1, \dots, N,
\eeq
we have
\begin{align}
\mathbb{E}[\| \nabla F(x^R) - z^R\|^2 ] &\le  \frac{1}{\sqrt{N}}\left[\frac{4L^2(T-1)\left[W_\gamma(x^0,z^0,w^0) + 2\sigma^2 \right]}{\lambda} +  2\prod_{\ell=1}^T\sigma_{J_{\ell}}^2 \right]\notag\\
&: = \frac{{\cal B}_2(\sigma^2)}{\sqrt{N}},\label{nasa_Fk-zk}\\
\mathbb{E}[V(x^R,z^R)] &\le  \frac{1}{\sqrt{N}}\left[{\cal B}_2(\sigma^2)+\frac{\max(1,\beta^2)}{\lambda}\left[W_\gamma(x^0,z^0,w^0) + 2\sigma^2 \right] \right],\label{nasa_main_theorem3}\\
\mathbb{E}[\| f_i(w_{i+1}^R) - w_i^R \|^2] & \leq  \frac{1}{\lambda\sqrt{N}}\left[W_\gamma(x^0,z^0,w^0) + 2\sigma^2\right] \qquad \qquad i = 1, \dots, T. \label{nasa_main_theorem4}
\end{align}
\end{theorem}

\vgap

\begin{proof}
We first show part (a). Noting \cref{def_zk}, we have
\[
    \nabla F(x^{k+1}) - z^{k+1} = (1-\tau_k)(\nabla F(x^k) - z^k) + \tau_k (\delta^k+\bar \delta^k +\Delta^k),
\]
where $\Delta^k$ is defined in \cref{def_R_error} and
\[
    \delta^k = \nabla F(x^k) - \nabla f_T(x^k) \prod_{i=2}^T \nabla f_{T+1-i}(w^k_{T+2-i}), \qquad
    \bar \delta^k = \frac{\nabla F(x^{k+1}) - \nabla F(x^k)}{\tau_k}.
\]
Denoting $\bar \Delta_k = \langle \Delta^k, (1-\tau_k)(\nabla F(x^k) - z^k) + \tau_k (\delta^k+\bar \delta^k) \rangle$, we have
\begin{align*}
&\|\nabla F(x^{k+1}) - z^{k+1}\|^2 \\
 =& \|(1-\tau_k)(\nabla F(x^k) - z^k) + \tau_k (\delta^k+\bar \delta^k)\|^2 +\tau_k^2 \|\Delta^k\|^2+2\tau_k \bar \Delta_k \\
\le & (1-\tau_k) \|\nabla F(x^k) - z^k\|^2 + 2\tau_k \left[\|\delta^k\|^2+L_{\nabla F}^2\|d^k\|^2+\bar \Delta_k\right]+\tau_k^2 \|\Delta^k\|^2,
\end{align*}
where the inequality follows from convexity of $\|\cdot\|^2$ and Lipschitz continuity of gradient of $F$. Thus, in the view of~\cref{Gamma_lemma}, we obtain
%
%
%
%
%
%
%
%
\[
    \| \nabla F(x^k) - z^k\|^2 \leq 2\Gamma_k \sum_{i=0}^{k-1}\frac{\tau_i}{\Gamma_{i+1}}\left(\|\delta^i\|^2+L_{\nabla F}\|d^i\|^2+\bar \Delta_i+\frac{\tau_i}{2}\|\Delta^i\|^2\right),
\]
which implies that $ \sum_{k=1}^N \tau_k \|\nabla F(x^k) - z^k\|^2 =$
\begin{align}
    &2\sum_{k=1}^N\tau_k\Gamma_k \sum_{i=0}^{k-1}\frac{\tau_i}{\Gamma_{i+1}}\left(\|\delta^i\|^2+L_{\nabla F}^2\|d^i\|^2+\bar \Delta_i+\frac{\tau_i}{2}\|\Delta^i\|^2\right) \notag\\
    =& 2\sum_{k=0}^{N-1} \frac{\tau_k}{\Gamma_{k+1}}\left(\sum_{i=k+1}^N \tau_i\Gamma_i\right)\left(\|\delta^k\|^2+L_{\nabla F}^2\|d^k\|^2+\bar \Delta_k+\frac{\tau_k}{2}\|\Delta^k\|^2\right)\notag\\
    \le & 2c \sum_{k=0}^{N-1}\tau_k\left(\|\delta^k\|^2+L_{\nabla F}^2\|d^k\|^2+\bar \Delta_k+\frac{\tau_k}{2}\|\Delta^k\|^2\right),
    \label{Fz-zk_error}
\end{align}
where the last inequality follows from~\cref{gamma_condition}.

Now, observe that under~\cref{assumption:original_nasa_assumption}, we have
\[
\mathbb{E}[\bar \Delta_k | \mathscr{F}_k] = 0, \qquad \mathbb{E}[\|\Delta_k\|^2 | \mathscr{F}_k] \le
\mathbb{E}\left[\left\| \prod_{\ell=1}^T J_{\ell}^{k+1} \right\|^2 \ \Bigg| \mathscr{F}_k\right] \le \prod_{\ell=1}^T\sigma_{J_{\ell}}^2.
\]
Moreover, by~\cref{FZ_difference} and the fact that $(\sum_{i=1}^n a_i)^2 \leq n\sum_{i=1}^n a_i^2$ for nonnegative $a_i$'s, we have
\begin{align*}
\|\delta_k\|^2 &=\left\| \nabla F(x) - \nabla f_T(x) \prod_{i=2}^T \nabla f_{T+1-i}(w_{T+2-i}) \right\|^2\\
& \leq  2(T-1) \sum_{j=2}^{T-1} C_j^2\|f_j(w_{j+1}) - w_j\|^2 + 2C_T^2\|f_T(x) - w_T\|^2.
\end{align*}
Combining the above observations with \cref{Fz-zk_error} and in the view of \cref{def_L}, we obtain
\begin{align}
    &\sum_{k=1}^N \tau_k \mathbb{E}[\|\nabla F(x^k) - z^k\|^2 | \mathscr{F}_k ] \le c \prod_{\ell=1}^T\sigma_{J_{\ell}}^2 \sum_{k=0}^{N-1} \tau_k^2  \notag \\
    &+ 4cL(T-1) \sum_{k=0}^{N-1}\tau_k\left( \sum_{j=2}^{T-1} \|f_j(w_{j+1}) - w_j\|^2 + \|f_T(x) - w_T\|^2+\|d^k\|^2\right)\notag
\end{align}
Then, \cref{nasa_main_theorem} follows from the above inequality, \cref{main_recursion}, and \cref{Rk_bnd}.
\vgap
\noindent Part (b) then follows from part (a), \cref{Lyaponuv_function2}, \cref{main_recursion}, and noting that
\[
\mathbb{E}[V(x^R,z^R)] = \frac{\sum_{k=1}^{N}\tau_k V(x^k,z^k)}{\sum_{j=1}^{N}\tau_j}.
\]
Part (c) also follows by noting that choice of $\tau_k$ in \cref{def_tau} implies that
\begin{align*}
    \sum_{k=1}^{N} \tau_k \geq \sqrt{N}, \quad \sum_{k=0}^{N} \tau_k^2 = 2, \quad \Gamma_k = \left(1-\frac{1}{\sqrt{N}}\right)^{k-1}, \\ \sum_{i=k+1}^{N}\tau_i\Gamma_i = \left(1- \frac{1}{\sqrt{N}}\right)^k\frac{1}{\sqrt{N}}\sum_{i=0}^{N-k-1}\left(1-\frac{1}{\sqrt{N}}\right)^i \leq \left(1-\frac{1}{\sqrt{N}}\right)^k,
\end{align*}
ensuring condition \cref{gamma_condition} with $c=1$.
%
%
%

\end{proof}
\begin{remark}\label{remark:rmk1}
The result in \eqnok{nasa_main_theorem3} implies that to find an $\epsilon$-stationary point of \eqnok{eq:mainprob} (see,~\cref{def_Vxz}), \cref{alg:originalnasa} requires ${\cal O}(\rho^T T^4/\epsilon^4)$ number of iterations, where $\rho$ is a constant depending on the problem parameters (i.e., Lipschitz constants and noise variances). Thus, the total number of used samples is bounded by
\[
\sum_{k=1}^T b_k = {\cal O}\left(\frac{\rho^T T^6} {\epsilon^6}\right) = {\cal O}_T\left(\frac{1} {\epsilon^6}\right)
\]
due to \eqnok{def_bk} and \eqnok{def_tau}. This bound is much better than $\mathcal{O}_T\left(1/\epsilon^{(7+T)/2}\right)$ obtained in \cite{yang2019multi-level} when $T >4$\footnote{Following the presentation in~\cite{yang2019multi-level}, we only present the $\epsilon$-related $T$ dependence for their result.} . In particular, it exhibits the level-independent behavior as discussed in~\cref{sec:intro}. Note that, we obtain constants of order $\rho^T$, for example, when $\sigma^2_{J_i}$ in~\cref{Rk_bnd} are all equal. We emphasize that~\cite{yang2019multi-level} and~\cite{zhang2019multi} also have such constant factors that depend exponentially on $T$, in their proofs and the final results.
\end{remark}

\begin{remark}\label{remark:rmk1-1}
The bound in \eqnok{nasa_main_theorem4} also implies that the errors in estimating the inner function values decrease at the same rate that we converge to the stationary point of the problem. This is essential to obtain a rate of convergence similar to that of single-level problems. Moreover, \eqnok{nasa_Fk-zk} shows that the stochastic estimate $z^k$ also converges at the same rate to the gradient of the objective function at the stationary point where $x^k$ converges to.
\end{remark}

Although our results for~\cref{alg:originalnasa} show improved convergence rates compared to~\cite{yang2019multi-level}, it is still worse than $\mathcal{O}_T\left(1/\epsilon^4\right)$ obtained in \cite{GhaRuswan20} for the case of $T=2$. Furthermore, the batch sizes $b_k$ is of order $\rho^T$ for some constant $\rho$ which makes it impractical. In the next section, we show that both of these issues could be fixed by a properly modified variant of \cref{alg:originalnasa}.

\section{Multi-level Nested Linearized Averaging Stochastic Gradient \\ Method}\label{sec:modifiednasa}
In this section, we present a linearized variant of \cref{alg:originalnasa} which can achieve the best known rate of convergence for problem \eqnok{eq:mainprob} for any $T\geq 1$, under Assumptions~\ref{fi_lips} and~\ref{assumption:original_nasa_assumption}. Indeed, when $T>2$, we have accumulated errors in estimating the inner function values. Hence, in~\cref{alg:originalnasa} we use mini-batch sampling in \eqnok{def_wi} to reduce the noise associated with the stochastic function values. However, this increases the sample complexity of the algorithm. To resolve this issue, instead of using the point estimates of $f_i$'s, we use their stochastic linear approximations in \eqnok{def_wi_new}. This modification reduces the bias error in estimating the inner function values which together with a refined convergence analysis enable us to obtain a sample complexity of ${\cal O}_T(1/\epsilon^4)$ with \cref{alg:modifiednasa}, for any $T\geq 1$ without using any mini-batches. \textcolor{black}{Furthermore,~\cref{alg:modifiednasa} works with any constant choice of step-size parameter $\beta$ (independent of the problem parameters), making it easy to implement.} As mentioned previously,~\cref{alg:modifiednasa} is motivated by the algorithm in \cite{rusz20} proposed for solving nonsmooth multi-level stochastic composition problems. However, \cite{rusz20} assumes that all functions $f_i$ explicitly depend on the decision variable $x$ which makes the composition function a variant of the general case in \eqnok{eq:mainprob}. It is also worth mentioning that other linearization techniques have been used in~\cite{duchi2018stochastic, davis2019stochastic} in estimating the stochastic inner function values for weakly convex two-level composition problems.


\begin{algorithm}
\caption{Multi-level Nested Linearized Averaging Stochastic Gradient Method}
\begin{algorithmic}
\STATE Set $b_k=1$ in Algorithm~\ref{alg:originalnasa} and replace \eqnok{def_wi} with
\beq\label{def_wi_new}
w_i^{k+1} = (1 - \tau_k)w_i^k + \tau_k G_i^{k+1} + (J_i^{k+1})^\top(w_{i+1}^{k+1} - w_{i+1}^k), \qquad  1\leq i \leq T.
\eeq
\end{algorithmic}
\label{alg:modifiednasa}
\end{algorithm}

To establish the rate of convergence of \cref{alg:modifiednasa}, we first need the next result which provides the recursion on the errors in estimating the inner function values.
\begin{lemma}\label{new_nasa_lemma_fn_w_error}
Let $\{x^k\}_{k \ge 0}$ and $\{w^k_i\}_{k \ge 0}$ be generated by~\cref{alg:modifiednasa}. Define, for $1\leq i\leq T$,
\begin{align}
e_i^{k+1}:= f_i(w_{i+1}^k) - G_i^{k+1},&~\hat{e_i}^{k+1}: = \nabla f_i(w_{i+1}^k) - J_i^{k+1},\label{new_nasa_errors}\\
\hat A_{k,i}  := f_i(w_{i+1}^{k+1}) - f_i(w_{i+1}^k&) - \nabla f_i(w_{i+1}^k)^\top (w_{i+1}^{k+1}-w_{i+1}^k) \label{new_nasa_A_ki}.
\end{align}
\begin{itemize}[leftmargin=0.2in]
\item [a)] Under~\cref{fi_lips}, we have, for~$1 \leq i \leq T$,
\begin{align}
 & \|f_i(w_{i+1}^{k+1}) - w_i^{k+1}\|^2  \leq (1-\tau_k)\|f_i(w_{i+1}^k) - w_i^k\|^2 + \tau_k^2\|e_i^{k+1}\|^2 + \dot{r}_i^{k+1} \notag\\
& + \left[8 L_{ f_i}^2+ L_{\nabla f_i} \|f_i(w_{i+1}^k) - w_i^k\|+ \| \hat{e_i}^{k+1}\|^2\right] \|w_{i+1}^{k+1}-w_{i+1}^k\|^2, \label{new_nasa_fn_w_error_i}\\
\dot{r}^{k+1}_i &:=2\tau_k\langle e_i^{k+1}, \hat A_{k,i} + (1-\tau_k)(f_i(w_{i+1}^k) - w_i^k) + (\hat{e_i}^{k+1})^\top(w_{i+1}^{k+1} - w_{i+1}^k)\rangle  \nonumber \\ 
&+2\langle (\hat{e_i}^{k+1})^\top(w_{i+1}^{k+1} - w_{i+1}^k), \hat A_{k,i}+(1-\tau_k)(f_i(w_{i+1}^k)-w_i^k)\rangle.\label{new_nasa_r_i}
\end{align}

\item [b)] Furthermore, we have for $1\leq i \leq T$, $\| w_i^{k+1} - w_i^k\|^2 \leq$
\begin{align*}
& \tau_k^2\left[2\| f_i(w_{i+1}^k) - w_i^k\|^2 + \|e_i^{k+1}\|^2 + \frac{2}{\tau_k^2}\|J_i^{k+1}\|^2 \| w_{i+1}^{k+1} - w_{i+1}^k \|^2\right]+2 \ddot{r}^{k+1}_i,
\end{align*}
where $ \ddot{r}^{k+1}_i := \tau_k \langle -e_i^{k+1}, \tau_k(f_i(w_{i+1}^k) - w_i^k) + (J_i^{k+1})^\top(w_{i+1}^{k+1}-w_{i+1}^k)\rangle.$
\end{itemize}
\end{lemma}
\begin{proof}
We first prove part a). 

When $1\leq i < T$, by definition of $\hat A_{k,i}, \hat{e_i}^{k+1},G_i^{k+1},w_i^{k+1}$, and $\dot{r}_i^{k+1}$, we have
\begin{align}
    & \|f_i(w_{i+1}^{k+1}) - w_i^{k+1}\|^2 \notag\\ 
    = & \|\hat A_{k,i} + f_i(w_{i+1}^k) + \nabla f_i(w_{i+1}^k)^\top (w_{i+1}^{k+1}-w_{i+1}^k) \notag  \\
    &\qquad- (1-\tau_k)w_i^k - \tau_kG_i^{k+1} - (J_i^{k+1})^\top(w_{i+1}^{k+1}-w_{i+1}^k)\|^2 \notag\\  
    = & \|\hat A_{k,i} + (\widehat{e_i}^{k+1})^\top(w_{i+1}^{k+1}-w_{i+1}^k) + (1-\tau_k)(f_i(w_{i+1}^k) - w_i^k) + \tau_ke_i^{k+1}\|^2 \label{fi_wi_modified_nasa}\\  
    =& \|(\widehat{e_i}^{k+1})^\top(w_{i+1}^{k+1}-w_{i+1}^k)\|^2 + \|\hat A_{k,i} + (1-\tau_k)(f_i(w_{i+1}^k)-w_i^k)\|^2 + \tau_k^2\|e_i^{k+1}\|^2 + \dot{r}_i^{k+1} \notag\\  
    \leq & \|\hat A_{k,i} + (1-\tau_k)(f_i(w_{i+1}^k) - w_i^k)\|^2 +\tau_k^2\|e_i^{k+1}\|^2 + \dot{r}_i^{k+1} + \|\widehat{e_i}^{k+1}\|^2\|w_{i+1}^{k+1}-w_{i+1}^k\|^2\notag\\
    \leq & (1-\tau_k)\|f_i(w_{i+1}^k) - w_i^k\|^2+\|\hat A_{k,i}\|^2+ 2(1-\tau_k)\langle \hat A_{k,i}, f_i(w_{i+1}^k) - w_i^k\rangle +\tau_k^2\|e_i^{k+1}\|^2 \notag\\
    &\qquad + \dot{r}_i^{k+1}+ \|\widehat{e_i}^{k+1}\|^2\|w_{i+1}^{k+1}-w_{i+1}^k\|^2.\label{fi_wi_modified_nasa2}
\end{align}
Now, noting that under~\cref{fi_lips}, we have 
\beq\label{Ak_bnd}
\|\hat A_{k,i}\| \le \frac{1}{2}\min\left\{4 L_{f_i}\|w_{i+1}^{k+1}-w_{i+1}^k\|, L_{\nabla f_i}\|w_{i+1}^{k+1}-w_{i+1}^k\|^2 \right\},
\eeq
and using Cauchy–Schwarz inequality in \eqnok{fi_wi_modified_nasa2}, we obtain \cref{new_nasa_fn_w_error_i}.

To show part b), noting definition of \cref{def_wi_new} and \cref{new_nasa_errors}, Cauchy-Schwartz and Young's inequality, for $1\leq i \leq T$,
\begin{align*}
    &\| w_i^{k+1} - w_i^k\|^2 \\
     =& \| \tau_k(G_i^{k+1}-w_i^k) + (J_i^{k+1})^\top(w_{i+1}^{k+1} - w_{i+1}^k)\|^2 \\  = &\tau_k^2\|G_i^{k+1} - w_i^k\|^2 + \| (J_i^{k+1})^\top(w_{i+1}^{k+1} - w_{i+1}^k)\|^2 + 2\tau_k\langle G_i^{k+1} - w_i^k, (J_i^{k+1})^\top(w_{i+1}^{k+1} - w_{i+1}^k)\rangle \\
    \leq & \tau_k^2 \| G_i^{k+1} - w_i^k\|^2 + 2\|J_i^{k+1}\|^2 \|w_{i+1}^{k+1} - w_{i+1}^{k}\|^2 + \tau_k^2 \|f_i(w_{i+1}^k) - w_i^k\|^2 \\
     &+ 2\tau_k\langle -e_i^{k+1}, (J_i^{k+1})^\top(w_{i+1}^{k+1} - w_{i+1}^k)\rangle  \\
     =& 2\tau_k^2 \| f_i(w_{i+1}^k) - w_i^k\|^2 + \tau_k^2 \|e_i^{k+1}\|^2 + 2\|J_i^{k+1}\|^2 \|w_{i+1}^{k+1} - w_{i+1}^k\|^2 \\  &+ 2\tau_k \langle -e_i^{k+1}, \tau_k(f_i(w_{i+1}^k) - w_i^k) + (J_i^{k+1})^\top(w_{i+1}^{k+1}-w_{i+1}^k)\rangle.
\end{align*}
\end{proof}

\noindent In the next result, we show how the moments of $\|w_i^{k+1} - w_i^k\|$ decrease in the corresponding order of $\tau_k$. This is a crucial step on bounding the errors in estimating the inner function values.

\begin{lemma}\label{moment_bounds}
Under~\cref{fi_lips}, \cref{assumption:original_nasa_assumption}, for $1\leq i \leq T$, and with the choice of $\tau_0=1$ (for simplicity), we have
\begin{align}
&\mathbb{E}[ \|f_i(w_{i+1}^{k+1}) - w_i^{k+1}\|^2  | \mathscr{F}_k] \leq  \sigma_{G_i}^2+(4L_{f_i}^2+ \hat \sigma_{J_i}^2)c_{i+1} ,\label{fi-wi_bound}\\
&\mathbb{E}[\|w_i^{k+1} - w_i^k\|^2 | \mathscr{F}_k] \leq c_i~\tau^2_k,  \label{new_nasa_conditional_expectation_w_T_second_power}
\end{align}
where, for $1 \le i \le T$,
\begin{align}
c_i := 3 \sigma_{G_T}^2+2(4L_{f_T}^2+\hat \sigma_{J_T}^2+\sigma_{J_T}^2) c_{i+1},~~~\text{with}~~~ c_{T+1} := \left(\prod_{i=1}^T\sigma_{J_i}^2\right) \beta^{-2}.\label{def_ci}
\end{align}
\end{lemma}

\vgap

\begin{proof}
Recall the definitions of $\hat A_{k,i}, e_i^{k+1}, \hat{e}_i^{k+1}$ and, for $1\leq i\leq T$, define
\begin{align}
    D_{k,i} := \hat A_{k,i} + \tau_ke_i^{k+1} + \hat{e}_i^{k+1}(w_{i+1}^{k+1}-w_{i+1}^k).\label{new_nasa_D_ki}
\end{align}
Then, by \eqnok{fi_wi_modified_nasa}, for $1\leq i\leq T$, we have
\begin{align}
    f_i(w_{i+1}^{k+1}) - w_i^{k+1} & = (1-\tau_k)(f_i(w_{i+1}^k)-w_i^k) + D_{k,i},\label{new_nasa_fi_w_i_D_ki}
\end{align}
which together with the convexity of $\|\cdot\|^2$, imply that
\begin{align}
\|f_i(w_{i+1}^{k+1})-w_i^{k+1}\|^2 &\leq (1-\tau_k)\|f_i(w_{i+1}^{k})-w_i^k\|^2 + \frac{1}{\tau_k}\|D_{k,i}\|^2. \label{fi-wi_Dki}\
\end{align}
Moreover, we have
\begin{align}
\|D_{k,i}\|^2 &= \|\hat A_{k,i}\|^2 + \tau_k^2 \|e_i^{k+1}\|^2 + \|(\hat{e}_i^{k+1})^\top(w_{i+1}^{k+1}-w_{i+1}^k)\|^2 +2 r'_{k,i}, \label{Dki_squared}\\
r'_{k,i} &= \langle \hat A_{k,i},\tau_k e_i^{k+1}+(\hat{e}_i^{k+1})^\top(w_{i+1}^{k+1}-w_{i+1}^k) \rangle + \tau_k \langle e_i^{k+1}, (\hat{e}_i^{k+1})^\top(w_{i+1}^{k+1}-w_{i+1}^k)\rangle,\nonumber
\end{align}
which together with the fact that $\mathbb{E}[r'_{k,i}| \mathscr{F}_k]=0$ under~\cref{assumption:original_nasa_assumption}, imply that
\begin{align}
\mathbb{E}[\|D_{k,i}\|^2| \mathscr{F}_k] &= \mathbb{E}[\|\hat A_{k,i}\|^2| \mathscr{F}_k] + \tau_k^2 \mathbb{E}[\|e_i^{k+1}\|^2| \mathscr{F}_k] +
\mathbb{E}[\|\hat{e}_i^{k+1}(w_{i+1}^{k+1}-w_{i+1}^k)\|^2| \mathscr{F}_k] \nonumber\\
&\le \tau_k^2 \mathbb{E}[\|e_i^{k+1}\|^2| \mathscr{F}_k] +
\left(4L_{f_i}^2+\mathbb{E}[\|\hat{e}_i^{k+1}\|^2|\mathscr{F}_k]\right) \mathbb{E}[\|w_{i+1}^{k+1}-w_{i+1}^k\|^2| \mathscr{F}_k],\label{exp_Dki_squared}
\end{align}
where the second inequality follows from \eqnok{Ak_bnd}.
Hence, noting \eqnok{dk_bnd}, $w_{T+1}^{k} =x^k$,  we have
\begin{align*}
    \mathbb{E}[\|D_{k,T}\|^2 | \mathscr{F}_k] & \leq \tau_k^2\left[\sigma_{G_T}^2+\left(4L_{f_T}^2+\hat \sigma_{J_i}^2\right)\left(\prod_{i=1}^T\sigma_{J_i}^2\right) \beta^{-2}\right].
\end{align*}
Using \eqnok{fi-wi_Dki} with $i=T$, the above inequality, and~\cref{Gamma_lemma} with the choice of $\tau_0=1$, we have
\beq\label{fT_squared_bnd}
\mathbb{E}[\|f_T(x^{k})-w_T^{k}\|^2 | \mathscr{F}_k] \leq \sigma_{G_T}^2+\left(4L_{f_T}^2+\hat \sigma_{J_i}^2\right)\left(\prod_{i=1}^T\sigma_{J_i}^2\right) \beta^{-2}.
\eeq
Moreover, by \cref{new_nasa_lemma_fn_w_error}.b) and under \cref{assumption:original_nasa_assumption}, we have $\mathbb{E}[\|w_{i+1}^{k+1}-w_i^k\|^2 | \mathscr{F}_k] \le $
\beq\label{wi_ineq}
 \tau_k^2\mathbb{E}\left[2\| f_i(w_{i+1}^k) - w_i^k\|^2 + \|e_i^{k+1}\|^2 + \frac{2}{\tau_k^2}\|J_i^{k+1}\|^2 \| w_{i+1}^{k+1} - w_{i+1}^k \|^2 \Big| \mathscr{F}_k\right],
\eeq
implying that
\beq
\mathbb{E}[\|w_T^{k+1}-w_T^k\|^2 | \mathscr{F}_k] \le \tau_k^2\left[3 \sigma_{G_T}^2+2(4L_{f_T}^2+\hat \sigma_{J_T}^2+\sigma_{J_T}^2)\left(\prod_{i=1}^T\sigma_{J_i}^2\right) \beta^{-2}\right].
\eeq
This completes the proof of~\cref{fi-wi_bound} and \cref{new_nasa_conditional_expectation_w_T_second_power} when $i=T$. We now use backward induction to complete the proof.  By the above result, the base case of $i=T$ holds. Assume that $\mathbb{E}[\|w_{i+1}^{k+1}-w_{i+1}^k\|^2|\mathscr{F}_k] \leq c_{i+1}\tau_k^2$ for some $1\leq i < T$. Hence, by \cref{Dki_squared} and under~\cref{assumption:original_nasa_assumption}, we have
\begin{align*}
\mathbb{E}[\|D_{k,i}\|^2|\mathscr{F}_k] \leq \tau_k^2[\sigma_{G_i}^2+(4L_{f_i}^2+ \hat \sigma_{J_i}^2)c_{i+1}],
\end{align*}
which together with \cref{Gamma_lemma}, imply that
\begin{align*}
    \mathbb{E}[\|f_i(w_{i+1}^{k})-w_i^{k}\|^2 | \mathscr{F}_k] \leq \sigma_{G_i}^2+(4L_{f_i}^2+ \hat \sigma_{J_i}^2)c_{i+1}].
\end{align*}
Thus, by \cref{wi_ineq}, we obtain
\begin{align*}
    \mathbb{E}[\|w_i^{k+1}-w_i^k\|^2|\mathscr{F}_k] & \leq \tau_k^2[3\sigma_{G_i}^2+2(4L_{f_i}^2+ \hat \sigma_{J_i}^2+\sigma^2_{J_i})c_{i+1}],
\end{align*}
which together with the definition of $c_i$ in \eqnok{def_ci}, complete the proof.
\end{proof}

\vgap

\noindent The next result is the counterpart of \cref{original_nasa_merit_function_lemma} for \cref{alg:modifiednasa}.

\begin{lemma} \label{new_nasa_merit_function_lemma}
Recall the definition of the merit function in~\cref{merit_function}. Define $w^k:= (w_1^k,\dots,w_T^k)$ for $k \geq 0$. Let $\{x^k,z^k,u^k,w_1^k,\dots,w_T^k\}_{k \geq 0}$ be the sequence generated by~\cref{alg:modifiednasa}. Suppose that
\begin{align}
\gamma_1 \ge \lambda >0, \quad \beta >\lambda, \quad    (\beta-\lambda ) (\gamma_j - \lambda) \ge 4 T C_j^2 , \qquad  j \in \{2,\ldots,T\}, \label{new_nasa_merit_function_upperbound_assumption}
\end{align}
where $C_j$'s are defined in Lemma~\ref{FZ_difference}. Then, under~\cref{fi_lips} and~\cref{assumption:original_nasa_assumption}, we have
\begin{align} \label{modified_nasa_main_rec}
&\lambda \sum_{k=0}^{N-1}\tau_k\left[\|d^k\|^2 + \sum_{i=1}^{T-1}\|f_i(w_{i+1}^k) - w_i^k\|^2 + \|f_T(x^k) - w_T^k\|^2 \right] \notag\\
&\leq W_\gamma(x^0,z^0,w^0) + \sum_{k=0}^{N-1} \hat{R}^{k+1},
\end{align}
where, for any $k \geq 0$,
\begin{align}
\hat{R}^{k+1} & := \left(\sum_{i=1}^T\gamma_i\hat{r}_i^{k+1}\right) + \frac{\tau_k^2}{2}\left[(L_{\nabla F} + L_{\nabla \eta} \right] + \tau_k\langle d^k,\Delta_k\rangle + \frac{L_{\nabla \eta}}{2}\|z^{k+1} - z^k\|^2, \label{def_Rhatk}\\
    \hat{r}_i^{k+1} & = \left[8 L_{ f_i}^2+ L_{\nabla f_i} \|f_i(w_{i+1}^k) - w_i^k\|+ \| \hat{e_i}^{k+1}\|^2\right] \|w_{i+1}^{k+1}-w_{i+1}^k\|^2 \notag\\
    &+\tau_k^2\|e_i^{k+1}\|^2 + \dot{r}_i^{k+1},\notag
\end{align}
and $\Delta_k$ and $\dot{r}_i^{k+1}$ are, respectively, defined in \eqnok{def_deltak} and \eqnok{new_nasa_r_i}.
Furthermore, notice that~\cref{new_nasa_merit_function_upperbound_assumption} is satisfied, when we pick
{\color{black}
\begin{align}\label{new_nasa_merit_function_gamma_choice}
\gamma_1 = \lambda = \sqrt{T}, \qquad \beta = 2\sqrt{T}, \qquad  \gamma_j = \sqrt{T}(1  + 4 C_j^2) \qquad 2\le j \le T.
\end{align}
}

\end{lemma}

\begin{proof}
Noting \cref{new_nasa_lemma_fn_w_error} and definition of $\hat{r}_i^{k+1}$, we have, $\forall i \in \{1,2,\ldots, T\},$
\begin{align*}
    \|f_i(w_{i+1}^{k+1}) - w_i^{k+1}\|^2 - \|f_i(w_{i+1}^k) - w_i^k\|^2 & \leq -\tau_k\|f_i(w_{i+1}^k) - w_i^k\|^2 + \hat{r}_i^{k+1}.
\end{align*}
Combining the above inequalities with \eqnok{F_lips}, \eqnok{eta_lips}, noting definition of the merit function in \eqnok{merit_function}, and in the view of \cref{FZ_difference}, we obtain
\begin{align*}
    & W_\gamma(x^{k+1},z^{k+1},w^{k+1}) - W_\gamma(x^k,z^k,w^k) \\ \leq & -\beta \tau_k\|d^k\|^2 + \sum_{j=2}^{T-1}\tau_kC_j\|d^k\| \|f_j(w_{j+1}^k)-w_j^k\| + \tau_kC_T\|d^k\|\|f_T(x^k)-w_T^k\| \\  -& \sum_{j=1}^{T-1}\gamma_j\tau_k\|f_j(w_{j+1}^k) - w_j^k\|^2 - \gamma_T\tau_k\|f_T(x^k) - w_T^k\|^2 + \hat{R}^{k+1}
\end{align*}
{\color{black}Now if condition \cref{new_nasa_merit_function_upperbound_assumption} holds, for any $i \in \{1,\ldots,T\}$, we have
 \begin{align*}
&-\frac{\beta}{T}\|d^k\|^2 + C_i\|d^k\| \|f_i(w_{i+1}^k)-w_i^k\|- \gamma_i\|f_i(w_{i+1}^k) - w_i^k\|^2 \\
\le & -\lambda \big[\frac{1}{T}\|d^k\|^2 + \|f_i(w_{i+1}^k) - w_i^k\|^2\big].
\end{align*}
Combining the above inequalities, we obtain
\begin{align*}
& W_\gamma(x^{k+1},z^{k+1},w^{k+1}) - W_\gamma(x^k,z^k,w^k) \\
 \leq& -\lambda \tau_k \Big[\|d^k\|^2 + \sum_{j=1}^{T-1}\|f_j(w_{j+1}^k) - w_j^k\|^2 + \|f_T(x^k)-w_T^k\|^2\Big] + \hat{R}^{k+1}.
\end{align*}
}
Thus, by summing up the above inequalities and re-arranging the terms, we obtain \eqnok{modified_nasa_main_rec}. Finally, it is easy to see that \cref{new_nasa_merit_function_upperbound_assumption} holds, by picking the parameters as in~\cref{new_nasa_merit_function_gamma_choice}.
\end{proof}

\vgap

In the next result, we show the error terms in the right hand side of \eqnok{modified_nasa_main_rec} is bounded in the order of $\sum_{k=1}^N \tau_k^2$ in expectation.

\begin{proposition} \label{new_nasa_big_proposition}
Let $\hat{R}^k$ be defined in \eqnok{def_Rhatk}. The, under \cref{assumption:original_nasa_assumption}, we have
\[
\mathbb{E}[\hat{R}^{k+1} | \mathscr{F}_k] \leq \hat{\sigma}^2 \tau_k^2, \qquad \forall k \geq 1, 
\]
where 
\begin{align}
    \hat{\sigma}^2 & := \sum_{i=1}^{T}\gamma_i\left(\left[8L_{f_i}^2+L_{\nabla f_i}\sqrt{\sigma_{G_i}^2+(4L_{f_i}^2+ \hat \sigma_{J_i}^2)c_{i+1}}+ \hat \sigma_{J_i}^2 \right]{c}_{i+1} + \sigma_{G_i}^2\right) \notag\\
    &+  \frac{1}{2\beta^2}\left(\prod_{i=1}^T\sigma_{J_i}^2\right)[(1+4\beta^2) L_{\nabla \eta}+L_{\nabla F}]. \label{new_nasa_big_proposition_sigma_squared}
\end{align}

\end{proposition}

\begin{proof}
Under \cref{assumption:original_nasa_assumption}, we have, for $1 \leq i \leq T$,
\[
\mathbb{E}[\Delta_k | \mathscr{F}_k] = 0, \quad \mathbb{E}[\dot{r}_i^{k+1}|\mathscr{F}_k] = 0, \quad \mathbb{E}[ \| \hat{e_i}^{k+1}\|^2|\mathscr{F}_k] \le \sigma_{G_i}^2, \quad \mathbb{E}[ \| {e_i}^{k+1}\|^2|\mathscr{F}_k] \le \hat \sigma_{J_i}^2.
\]
Moreover, by \cref{moment_bounds} and Holder's inequality, we have $\mathbb{E}[\|w_i^{k+1} - w_i^k\|^2 | \mathscr{F}_k] \leq c_i~\tau^2_k$ and
\begin{align*}
&\mathbb{E}[ \|f_i(w_{i+1}^{k+1}) - w_i^{k+1}\| \|w_i^{k+1} - w_i^k\|^2| \mathscr{F}_k] \\ 
\le& \left(\mathbb{E}[ \|f_i(w_{i+1}^{k+1}) - w_i^{k+1}\|^2 | \mathscr{F}_k]\right)^{\frac12} \mathbb{E}[\|w_i^{k+1} - w_i^k\|^2 | \mathscr{F}_k]\\
\leq&  c_i \sqrt{\sigma_{G_i}^2+(4L_{f_i}^2+ \hat \sigma_{J_i}^2)c_{i+1}} \tau^2_k
\end{align*}
The result then follows by noting the definition of $\hat \sigma^2$ in \cref{new_nasa_big_proposition_sigma_squared}
\end{proof}

\vgap

We are now ready to state the convergence rates via the following theorem.
\begin{theorem}\label{optimal_theorem}
Suppose that $\{ x^k, z^k \}_{k\geq 0}$ are generated by~\cref{alg:modifiednasa}, and~\cref{fi_lips} and~\cref{assumption:original_nasa_assumption} hold. Also assume the parameters satisfy~\cref{new_nasa_merit_function_gamma_choice} and the step sizes $\{\tau_k\}$ satisfy \eqnok{gamma_condition}.
\begin{itemize}
\item [(a)] The results in parts a) and b) of \cref{nasa_main_theorem} still hold by replacing $\sigma^2$ by $\hat \sigma^2$.

\item [b)] If $\tau_k$ is set as in \eqnok{def_tau}, the results of part c) of \cref{nasa_main_theorem} also hold with $\hat \sigma^2$ replacing $\sigma^2$.
\end{itemize}

\end{theorem}

\begin{proof}
The proof follows from the same arguments as in the proof of \cref{nasa_main_theorem} by noticing \eqnok{modified_nasa_main_rec}, and \cref{new_nasa_big_proposition}, hence, we skip the details.

\end{proof}
\begin{remark}\label{remark:rmk2}
Note that \cref{alg:modifiednasa} does not use a mini-batch of samples in any iteration. Thus, \eqnok{nasa_main_theorem3} (in which $\sigma^2$ is replaced by $\hat \sigma^2$) implies that the total sample complexity of \cref{alg:modifiednasa} for finding an $\epsilon$-stationary point of \cref{eq:mainprob}, is bounded by ${\cal O}(c^T T^4/\epsilon^4)={\cal O}_T(1/\epsilon^4)$ which is better in the order of magnitude than the complexity bound of \cref{alg:originalnasa}. Furthermore, this bound matches the complexity bound obtained in \cite{GhaRuswan20} for the two-level composite problem which in turn is in the same order for single-level smooth stochastic optimization. Finally, it is worth noting that this complexity bound for \cref{alg:modifiednasa} is obtained without assuming boundedness of the feasible set or any dependence of the parameter $\beta$ on Lipschitz constants. Indeed, $\beta$ can be set to any positive number in the order of ${\cal O}(\sqrt{T})$ due to \eqnok{new_nasa_merit_function_gamma_choice}, and $\tau_k$ depends only on the total number of iterations $N$ due to \eqnok{def_tau}. This makes \cref{alg:modifiednasa} parameter-free and easy to implement.
\end{remark}

\section{Numerical Experiments}\label{sec:exp}
In this section, we provide numerical results for the risk-averse stochastic optimization problem introduced in Section~\ref{sec:motivatingeg}. The link function $g$ is set to be the square function and $U(x, \xi) \coloneqq -(b - g(a^\top x))^2$. In this case,~\eqref{eq:mdriskaverse} becomes a non-convex stochastic three-level composition optimization problem. For our experiments, we assume $a\in \mathbb{R}^d$ is a zero-mean Gaussian random vector with covariance matrix $\Sigma_{j,k} = 0.5e^{-\frac{|j-k|}{d}}$, following~\cite{yang2019multi-level}. Furthermore, $\zeta$ is a standard normal random variable. The true parameter $x^* \in \mathbb{R}^d$ is drawn from a standard Gaussian distribution and fixed. 

We compare our Algorithm~\ref{alg:modifiednasa} with the accelerated T-level stochastic compositional gradient descent (a-TSCGD) from~\cite{yang2019multi-level}. For our algorithm, the parameter $\tau_k$ was set at $c/\sqrt{N}$ (with $c$ being 0.5, 1 and 1.5) and the step-size $\beta$ was set to 4 (as it is close to $2\sqrt{T}$ and empirically worked the best). The parameters for a-TSCGD were set according to the suggestion from~\cite{yang2019multi-level}. We estimated the expected gradient size empirically, based on an independent dataset of size 10,000, so as to reduce any fluctuations in this estimation process. Furthermore, we reported the average over 100 Monte-Carlo trails, to reduce the fluctuations over the data generating process. Figure~\ref{fig:expcomparison} plots the empirical gradient norm squared as a function of iteration, for the values of dimension $d \in \{100, 500, 1000\}$. As can be seem from the plots, Algorithm~\ref{alg:modifiednasa} outperforms a-TSCGD from~\cite{yang2019multi-level} numerically as well. Furthermore, our algorithm is almost  insensitive to the choice of $c$ in the definition of $\tau_k$. 
\begin{figure}[t]\label{fig:expcomparison}
\centering
\includegraphics[scale=0.38]{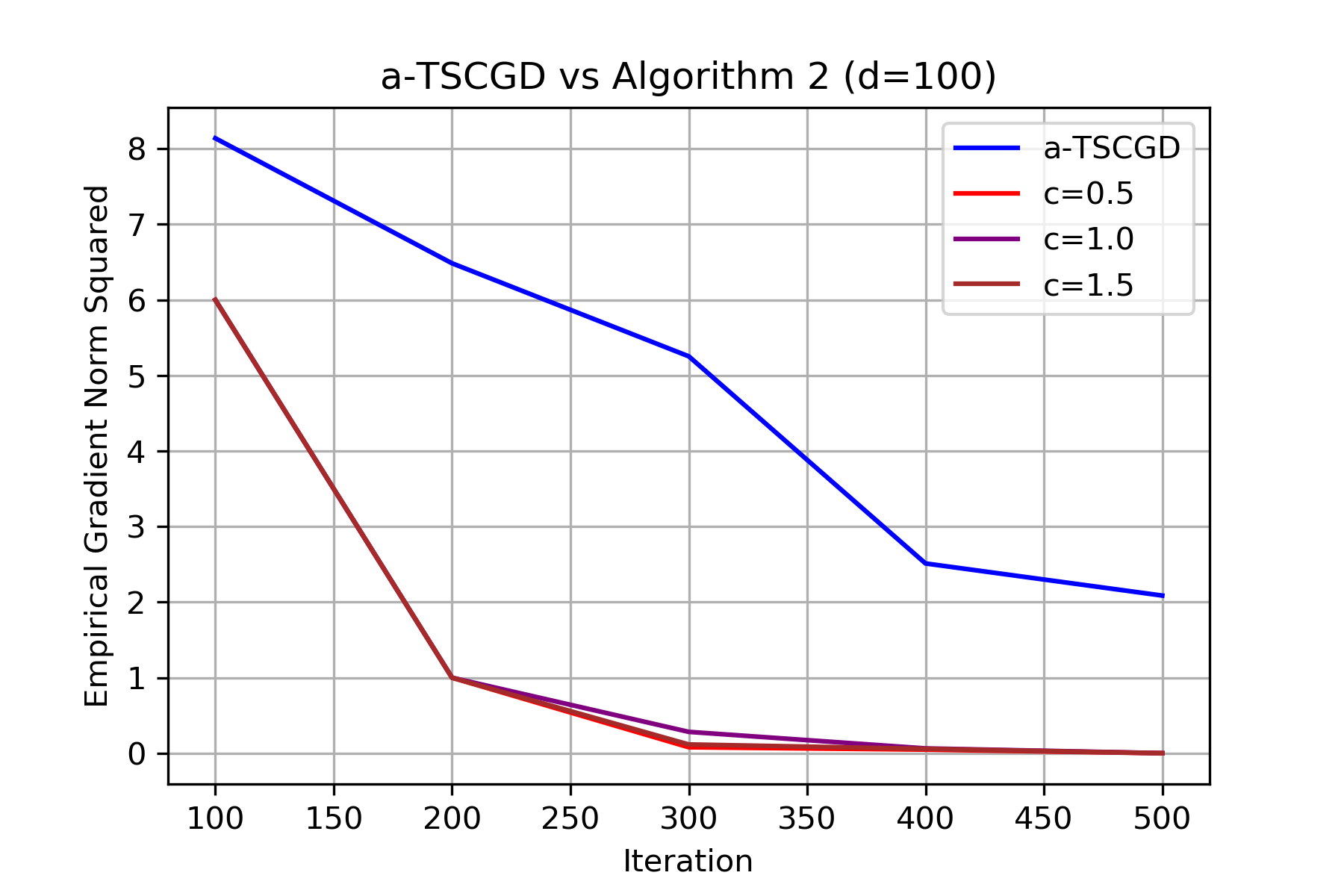}
\includegraphics[scale=0.38]{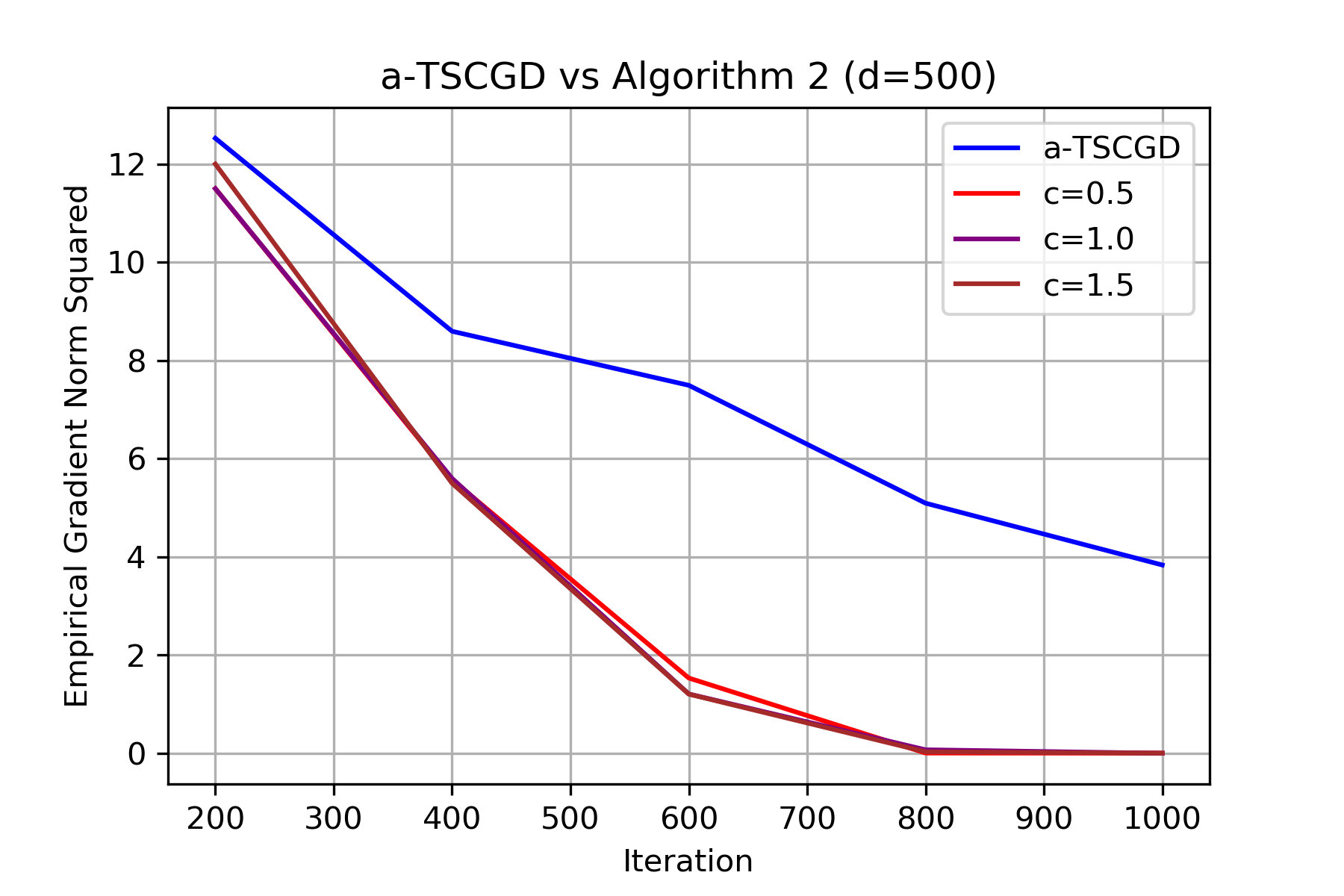}\\
\includegraphics[scale=0.38]{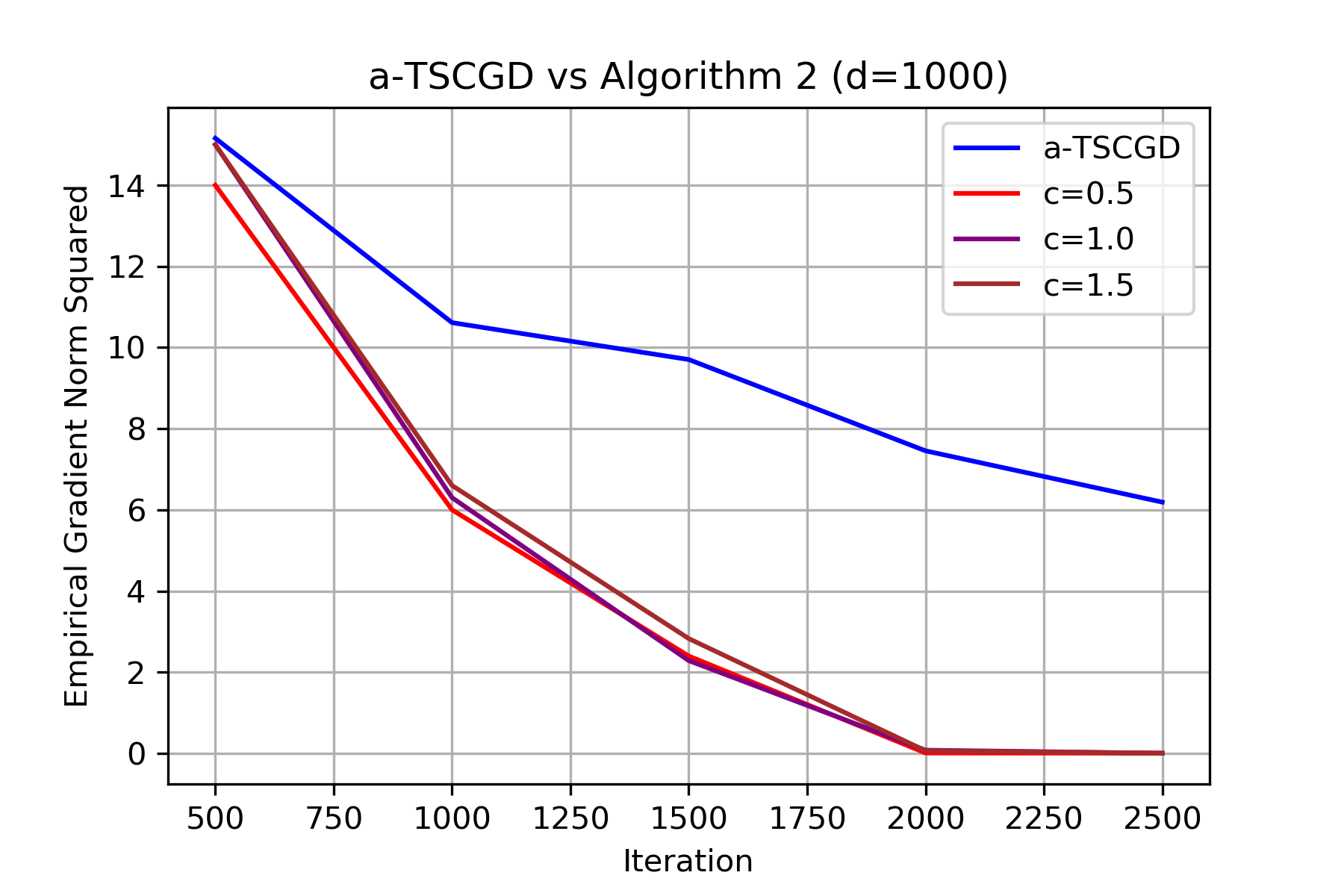}
\caption{Comparison between Algorithm~\ref{alg:modifiednasa} and a-TSCGD~\cite{yang2019multi-level}: Empirical gradient size squares versus iterations for $d=100$ (top left), $d=500$ (bottom) and $d=1000$ (top right). Here, $c$ refers to the choice of numerator in the tuning parameter $\tau_k$, given by $\tau_k\coloneqq: c/\sqrt{N}$.}
\end{figure}

\section{Concluding remarks}\label{conc_sec}
In this paper, we proposed two algorithms, with level-independent convergence rates, for stochastic multi-level composition optimization problems under the availability of a certain stochastic first-order oracle. We show that under a bounded second moment assumption on the outputs of the stochastic oracle, our first proposed algorithm, by using a mini-batch of samples in each iteration, achieves a sample complexity of ${\cal O}_T(1/\epsilon^6)$ for finding an $\epsilon$-stationary point of the multi-level composite problem. By modifying this algorithm with a linearization technique, we show that we can improve the sample complexity to ${\cal O}_T(1/\epsilon^4)$ which seems to be unimprovable even for single-level stochastic optimization problems, without further assumptions~\cite{Drori2019TheCO, arjevani2019lower}.
For future work, it would be interesting to establish CLT and normal approximation results for the online algorithms we presented in this work for stochastic multi-level composition optimization problems, similar to the results in~\cite{ruppert1988efficient, polyak1992acceleration, anastasiou2019normal, dieuleveut2020bridging, yu2020analysis} for the standard stochastic gradient algorithm when $T=1$.
\subsection*{Acknowledgment}
We would like to thank Prof. Andrzej Ruszczy{\'n}ski for helpful comments on the related works. Krishnakumar Balasubramanian was supported in part by UC Davis CeDAR (Center for Data Science and Artificial Intelligence Research) Innovative Data Science Seed Funding Program. Saeed Ghadimi was supported in part by an NSERC Discovery Grant.

\end{document}